\documentclass[12pt]{amsart}
\usepackage{amsmath, amsthm, amsfonts, amssymb, color}
\usepackage{mathrsfs}
\usepackage{amstext,amsxtra}
\usepackage{txfonts} %also pxfonts
\usepackage[colorlinks, citecolor=blue,pagebackref,hypertexnames=false]{hyperref}
\allowdisplaybreaks
\usepackage{pgf,tikz}

\newtheorem{theorem}{Theorem}[section]
\newtheorem{thm}{Theorem}
 
\newtheorem{proposition}{Proposition}[section]
\newtheorem{coro}{Corollary}[section]
\newtheorem{lemma}{Lemma}[section]
\newtheorem{definition}{Definition}[section]

\newtheorem{remark}{Remark}[section]

\newcommand\R{\mathbb{R}}
\newcommand\M{\mathcal{M}}
\renewcommand\o{\omega}
\renewcommand\S{\sigma}
\renewcommand{\r}{\right}
\renewcommand{\l}{\left}
\newcommand\E{\mathbb{E}}
\newcommand{\e}{\epsilon}

\begin{document}

\title{Vertical Maximal functions on manifolds with ends}
\author{Himani Sharma,  Adam Sikora} 
\address{Himani Sharma, Department of Mathematics, Macquarie University, NSW 2109, Australia \newline
\indent{\it Current address:} Institute for Analysis, Karlsruhe Institute of Technology, Karlsruhe, Germany}
\email{himani.sharma@kit.edu}
\address{
	Adam Sikora, Department of Mathematics, Macquarie University, NSW 2109, Australia}
\email{adam.sikora@mq.edu.au}


\begin{abstract}
We consider the setting of manifolds with ends which are obtained by compact perturbation (gluing)
of ends of the form $\R^{n_i}\times \mathcal{M}_i$. 
We investigate the family of vertical resolvent $\{\sqrt{t}\nabla(1+t\Delta)^{-m}\}_{t>0}$, where $m\geq1$. We show that 
the family is uniformly continuous on all $L^p$ for $1\le~p~\le~\min_{i}n_i$.
Interestingly, this is a closed-end condition in the considered setting. We prove that the corresponding maximal function is bounded in the same range except that it is only weak-type $(1,1)$ for $p=1$. The 
Fefferman-Stein vector-valued maximal function is again of weak-type $(1,1)$ but bounded if and only if $1<p<\min_{i}n_i$, and not at $p=\min_{i}n_i$. 
\end{abstract}

\maketitle

\section{Introduction}
Maximal functions play a significant role in the history and development of Harmonic Analysis,
see \cite[Chapters II and V]{St}. This research area is closely related to the notion of square functions, Riesz transform, and other singular integrals. Here our starting point is the theory of vertical maximal functions and we study the same setting for which the square functions were investigated in \cite{BaSi}. This is a situation which involves manifolds with ends and includes examples of spaces of non-doubling class. The description of Riesz transform continuity in this setting was provided in \cite{CCH} and \cite{HS2}, see also \cite{Ca2}.

This point motivates and initiates our study here in understanding the behaviour of certain maximal functions on $L^p$ spaces.

There are two main themes for which we believe readers could find our manuscript interesting:
\begin{enumerate}
	\item[$\bullet$] Maximal functions and Fefferman-Stein vector-valued maximal functions.
	\item[$\bullet$] Uniform continuity of gradient of heat kernels or resolvents and its relation with Riesz transform. 
\end{enumerate}
Let us recall
$L^{p_0}$-boundedness condition of the gradient of the heat semigroup generated by the Laplace-Beltrami operator $-\Delta$
$$\|\;|\sqrt t\nabla \exp(-t\Delta)|\;\|_{p_0\to p_0}\le C,\;\;\; \text{ see $(G_p)$-condition in \cite{ACDH}.} \leqno (G_{p_0}) $$
%see for example \cite{ACDH}.
%Note that  $(G_2)$ always holds.

We do not study vertical maximal functions and related ideas in their full generality but rather in illuminating 
setting of {\it manifolds with ends} which provides several surprising and inspiring examples illustrating the above two points.  It became clear from the beginning of our study that we should investigate these points in the 
context of two other central topics of harmonic analysis:  
\begin{enumerate}
	\item[$\bullet$] Square functions: We are interested in the following vertical square functions
\begin{equation}\label{VSF}
\Bigg(\int_0^\infty|\sqrt{t}\nabla\psi(t\Delta) x|^2\frac{dt}{t}\Bigg)^{1/2},
\end{equation}
where $\psi(t\Delta)$ can either be the semigroup generated by $-\Delta$ for $\psi(z)=e^{-z}$ or a resolvent operator, that is, $\psi(z):=(1+z)^{-m}$ for some $m\geq1$.
	\item[$\bullet$]  R-boundedness: We refer to Section \ref{R} for the description.
\end{enumerate}

Our investigation sheds light on the significance of results described in \cite{ABS2}
as well as suggests  several interesting future directions in {\it Analysis in Banach spaces.}

We consider here the dichotomy of 
\begin{enumerate}
	\item[$\bullet$] Horizontal maximal functions, see \eqref{HMF}.
	\item[$\bullet$]  Vertical maximal functions, see \eqref{VMF}.
\end{enumerate}
The horizontal variant involves only a generator of the semigroup and resolvent. 
The more unexpected behaviour is related to vertical type estimates which involve gradient action, which we discuss in detail below. The vertical estimates are the most interesting part of our discussion.

The classical semigroup theory is based on the equivalence between the resolvent and semigroup approach, see Pazy \cite{Pazy}. We believe that in the setting of manifolds with ends the resolvent approach is more natural and effective.
It is likely to be a better approach in studying compact perturbation of the standard Laplace operator and many similar settings. 
Thus, building harmonic analysis theory based on resolvent rather than on heat kernel is one of the directions which we would like to investigate here. A more comprehensive approach to the resolvent based harmonic analysis requires more extensive work going beyond our discussion here, and which we are going to study 
 somewhere else.

Our study here is devoted to the setting of manifolds with ends. In general, these spaces usually do not satisfy the doubling condition, and this is a very significant aspect of our investigation. In a metric measure space $(X,d,\mu)$, the {\it doubling condition} means that there exists a constant $C>0$ such that
\begin{equation}\label{Doubling}
    \mu(B(x,2t))\leq C\mu(B(x,t)) \tag{D}
\end{equation}
for all $x\in X$ and $t>0$, where the notation $B(x,t)$ is used to denote the ball of radius $t$ centered at the point $x$.

We start with the classical definition of maximal functions introduced by Hardy and Littlewood in 1930. The following Hardy-Littlewood maximal operator given in terms of the averages over balls $B(x,t)$ has been an essential part of harmonic analysis:
\begin{equation*}
    Mf(x):=\sup_{t>0}\frac{1}{\mu(B(x,t))}\int_{B(x,t)}|f(y)|dy.
\end{equation*}
The classical result shows the boundedness of $M$ on $L^p(\R^d)$ spaces for $1<p\leq\infty$ along with weak-type (1,1) bounds. %Mathematicians have studied the boundedness of this operator in different forms including their weighted versions on different spaces {\color{red} (Give references)}.  
These maximal functions are studied in various forms till date and Stein's book \cite{St} is a great source to know some classical results on them. 

The following Stein's maximal function 
\begin{equation*}
    M^{exp}f(x):=\sup_{t>0}|e^{-t\Delta}f(x)|
\end{equation*}
defined using the semigroup generated by the Laplace-Beltrami operator $-\Delta$ is of significant importance in the characterisation of Hardy spaces, see for instance \cite{BGS}. On Euclidean spaces, this maximal function is dominated by the Hardy-Littlewood maximal operator and hence is bounded on $L^p$ spaces with weak-type (1,1) estimates on $L^1$. In \cite[Page 73]{St_topics}, Stein showed the boundedness of $M^{exp}$ on $L^p$ spaces for $1<p\leq\infty$ implicitly, where he proved a more general result for the semigroup generated by a self-adjoint operator that satisfies contraction property $\|e^{-t\Delta}f\|_{L^p}\le \|f\|_{L^p} $. Similar result was obtained by Le Merdy and Xu in \cite[Corollary 4.2]{MeXu} for the derivative of the heat semigroup. The weak-type estimates were used to be obtained with the help of the standard Calder{\'o}n-Zygmund theory of 70's and 80's, specifically on the Euclidean spaces or more generally on doubling metric measure spaces. This theory no longer works in non-doubling situations.

In certain non-doubling situations, more precisely on manifolds with ends, Duong-Li-Sikora \cite{DLS} showed the weak-type (1,1) estimates for both Hardy-Littlewood maximal operator and Stein's maximal operator. %and $L^p$ boundedness of both $M$
%and $M^{exp}$ for $1<p\leq \infty$. 
These non-doubling complete Riemannian manifolds are formed as connected sums and are of great interest to us in this article. 

We start with a definition of a manifold with finitely many ends. We refer readers 
to \cite{GIS-C, GS1, GS2} for a more detailed description of the setting. Here we use notations similar to \cite{BaSi}.

\begin{definition}\label{ConnectedSum}
We say that a manifold $\mathcal{V}$ is a connected sum of a finite number of complete and connected manifolds $\mathcal{V}_1,\cdots,\mathcal{V}_l$ of the same dimension, denoted by
\begin{equation*}
\mathcal{V}=\mathcal{V}_1\#\mathcal{V}_2\#\cdots\#\mathcal{V}_l
\end{equation*}
if there exists some compact subset $K\subset \mathcal{V}$, with non-empty interior, for which $\mathcal{V}\setminus K$ can be expressed as the disjoint union of open subsets $E_i\subset \mathcal{V}$ for $i=1,\cdots,l$, where $E_i\simeq \mathcal{V}_i\setminus K_i$ for some compact subset $K_i\subset \mathcal{V}_i$.
\end{definition}

%\subsection{Manifolds with Ends}\label{SecManifolds}
Fix $N\in\mathbb{N}^*=\mathbb{N}\setminus\{0\}$. We consider an $N$-dimensional complete Riemannian manifold $ \mathcal{M}$ (unless otherwise stated) that is obtained  by taking the connected sum of $l \geq 2$
copies of manifolds which are Cartesian products of Euclidean spaces $\R^{n_i}$ with compact  Riemannian manifolds $\mathcal{M}_i$ and $n_i+\text{dim }\M_i=N$. That is, we are interested in smooth Riemannian manifolds of the form 
\begin{equation}\label{DefManifold}
    \M:=(\R^{n_1}\times\M_1)\#\cdots\#(\R^{n_l}\times\M_l).
\end{equation}
As mentioned in the previous definition, it is possible to choose compact subset $K\subset\M$ with non-empty interior, and open subsets $E_i\subset\M$ such that $\M\setminus K$ can be expressed as a disjoint union of $E_i$. This makes $\M$ a manifold with ends where $E_i$ are referred to as the ends and $K$ is considered to be the center of $\M$. The Riemannian structure on $\M$ restricted to each end coincides with the Cartesian product of Euclidean spaces $\R^{n_i}$ and some metric on $\M_i$. For more geometrical interpretation of these manifolds, see  \cite{GIS-C, GS1, GS2}.

Note that, an interesting case occurs where the dimensions $n_i$ of the Euclidean spaces are not all the same. That is, the ends have different `asymptotic dimension'. In that situation the manifold is not a doubling space. We can think of this intuitively as a ``connected sum of Euclidean spaces of different dimensions". Such a class of manifolds was studied in detail by Grigor'yan and Saloff-Coste, who obtained upper and lower bounds on the heat kernel on such manifolds, see \cite{GS1, GS2}. Duong-Li-Sikora used this to obtain weak-type (1,1) estimates for the Hardy-Littlewood maximal function and Stein's maximal function. At this point, no bounds similar to \cite{GS1,GS2} are known for the gradient as well as the derivative of the heat kernel in the setting of manifolds with ends. 

Therefore, in this article, we investigate the boundedness of the following resolvent based horizontal and vertical maximal functions, respectively, in the class of manifolds with ends. Here $\Delta$ is the negative Laplace-Beltrami operator on $\M$ and $\nabla$ is the gradient corresponding to the Riemannian structure
\begin{equation}\label{HMF}
M_m^{res,\Delta}f(x)=\sup_{t>0}|t\Delta (1+t\Delta)^{-m}f(x)|,\;\;\; m\in\mathbb{N},
\end{equation}
\begin{equation}\label{VMF}
    M_m^{res,\nabla}f(x)=\sup_{t>0}|\sqrt{t}\nabla (1+t\Delta)^{-m}f(x)|,\;\;\;m\in\mathbb{N}.
\end{equation}
For \eqref{HMF} and \eqref{VMF} we are able to use pointwise estimates from \cite{HS2}, see also Propositions \ref{PropH3} and \ref{PropH4} below. 
Our main aim is to prove the following theorems in the setting of Riemannian manifolds with ends. Thus, we assume $\M$ to be a manifold with ends defined in \eqref{DefManifold} with $n_i\geq3$ and $n^* =\min n_i$. We start our discussion here with a version of $(G_p)$-condition which should be connected with the results from \cite{ACDH}.

\begin{thm}\label{GpThm}
	Let $m\in\mathbb{N}$. The following estimates on the gradient of the resolvent operator holds:
	\begin{equation} \label{sc1}\|\;|{\sqrt t} \nabla (1+t\Delta)^{-m}|\;\|_{n^*\to n^*}\le {C} \quad \forall{t>0},
	\end{equation}
	\begin{equation} \label{sc2}\|\;|{\sqrt t} \nabla (1+t\Delta)^{-m}|\;\|_{1\to 1}\le {C} \quad \forall{t>0},
	\end{equation}
	and
	\begin{equation} \label{szesc}
	\|\;|{\sqrt t} \nabla (1+t\Delta)^{-m}|\;\|_{p \to p} \le  {C}  (\sqrt t)^{1-n^*/p}   \quad \forall{t\geq1}
	\end{equation}
for all $  p\in [n^*, \infty]$. 
Moreover for some $c>0$
\begin{equation} \label{szescNeg}
	\|\;|{\sqrt t} \nabla (1+t\Delta)^{-1}|\;\|_{p \to p} \ge  {c}  (\sqrt t)^{1-n^*/p}   \quad \forall{t\geq1}
	\end{equation}
for all $  p\in [n^*, \infty]$. 
	\end{thm}
 \begin{remark}
 Note that  \eqref{sc1} and \eqref{sc2} are similar to the estimates valid for the standard Laplace operator on Euclidean space $\R^n$. However, from \eqref{szesc} and \eqref{szescNeg}
 it follows that  for $t \ge 1$ the norm  $ \|\;|{\sqrt t} \nabla (1+t\Delta)^{-1}|\;\|_{p \to p}$ is comparable to $  (\sqrt t)^{1-n^*/p}$ 
 for $  p\in [n^*, \infty]$. For $p>n^*$ this asymptotic   essentially differs from the standard Laplace operator behaviour. 
 \end{remark}
The next statement is a maximal function version of the condition $(G_p)$. It is valid in the same range of $p$ except of weak-type $(1,1)$ for $L^1$.
\begin{thm}\label{maxThm}
Assume that $\M$ is a manifold with ends defined in \eqref{DefManifold}. 
Then on $L^p(\M)$
  \begin{enumerate}
  \item[(a)] The vertical maximal operator $M_m^{res,\nabla}$ is bounded if and only if 
	$1< p \le n^*$, that is,
	$$
	\|M_m^{res,\nabla}f(x)\|_{p \to p } \le C < \infty 
	$$
	if and only if $1<p \le  n^*$ (for some constant $C>0$). 
    \item[(b)] The horizontal  maximal operator $M^{res,\Delta}_m$
is bounded for all $1<p \leq \infty$, that is, 
		$$
	\|M_m^{res,\Delta}f(x)\|_{p \to p } \le c  < \infty 
	$$
for all  $1<p \leq \infty$ (for some constant $c>0$).
\end{enumerate}
Both operators $M_m^{res,\nabla}$ and $M_m^{res,\Delta}$ are also of weak-type $(1,1)$. 
\end{thm}

\begin{remark}
    The negative result in the case of the vertical maximal operator comes as a consequence of the resolvent based $(G_p)$ condition proved in Theorem~\ref{GpThm}. Interestingly, $(G_p)$ condition is not open ended unlike the one with the gradient of the heat kernel given in \cite{ACDH}. Note that in the setting considered in  \cite{CJKS}
    condition $(G_p)$ is proven to be open ended. 
\end{remark}
\begin{remark}
    It is known that $L^p$ continuity  for square  function \eqref{VSF} implies $(G_p)$ type version of conditions 
    \eqref{sc1} and \eqref{sc2} for $L^p$ spaces, see \cite[Proposition 5.3]{CO20} and \cite{CuDu}. Such square function estimates for 
    all $1<p<n^*$ were obtained in \cite{BaSi}. It was also verified that the square function estimates are not valid for $n^* \le p$. It means that this approach cannot be used to
    verify  \eqref{sc1} for $p=n^*$. The results described in \cite{BaSi} cannot lead to \eqref{szesc} and \eqref{szescNeg}.
    In \cite{CO20}, the following problem is also stated: ``One might ask whether boundedness of Square function in \eqref{VSF} is in turn equivalent to $(G_p)$. To the best of our knowledge, this question is also open in general." One can see that in the setting which we consider here the equivalence is not valid for 
    $p=n^*$.
\end{remark}

The next theorem gives us the vector-valued version of Theorem~\ref{maxThm}. Similar estimate for the Hardy-Littlewood maximal function was obtained by Fefferman and Stein in \cite{F-S} on Euclidean spaces. The equation~\eqref{FSeqn} below is a variant of the Fefferman-Stein maximal inequality for the vertical maximal operator. There are many such results for the Hardy-Littlewood maximal operator on different spaces. See, for instance, \cite[Theorem~1.2]{GLY} on spaces of homogeneous type and \cite[Theorem~3.1]{DKK} again on spaces of homogeneous type but for the lattice valued maximal function. The inequality in \eqref{FSeqn} for the vertical maximal operator comes naturally on a doubling manifold as a result of the Hardy-Littlewood maximal inequality and is mentioned in Proposition~\ref{FS}. In Theorem~\ref{FSThm} below we show this result on manifolds with ends.

\begin{thm}\label{FSThm}
	Consider the vertical maximal operator
    $M^{res,\nabla}_m$ on $L^p(\M)$. There exists a constant $A_{p,\M}>0$ (depending on $p$ and $\M$) such that
	\begin{align}\label{FSeqn}
	\bigg|\bigg|\bigg(\sum_{i=1}^\infty(M_m^{res,\nabla}f_i)^2\bigg)^{1/2}\bigg|\bigg|_{L^p}
	\leq A_{p,\M}\bigg|\bigg|\bigg(\sum_{i=1}^\infty|f_i|^2\bigg)^{1/2}\bigg|\bigg|_{L^p}
	\end{align}
	for all $1<p<n^*$ and of weak-type $(1,1)$.
	The estimates are false for $n^* \le p$.
\end{thm}

Another crucial idea for Harmonic analysis which is closely related to our discussion of Maximal 
functions is continuity of the Riesz transform acting on $L^p$ spaces:
$$\|\;|\nabla {\Delta}^{-1/2} f|\;\|_p\le C \|f\|_p\;\; \forall f\in L^p(X,\mu)\leqno (R_p)$$
and some constant $C>0$.
This is a basic issue which was already initiated in the setting of complete non-compact manifolds in \cite{Str} by Strichartz. It was proved in \cite{HS2} that  the Riesz transform  $T=\nabla \Delta^{-1/2}$ acts as a bounded operator from $L^p(\mathcal{M}) \to L^p(\mathcal{M}; T\mathcal{M})$ if and only if $1<p<n^*$ and that $T$ is also of weak-type $(1,1)$  in this class of manifolds. The results described in \cite{HS2} which includes non-doubling settings were continuation of studies developed in 
\cite{Ca2} and \cite{CCH}. 

An interesting application of the Riesz transform was given by Cometx and Ouhabaz in \cite{CO20}. They showed that the boundedness of the Riesz transform implies R-boundedness of the set $\{\sqrt{t}\nabla e^{-t\Delta}:t>0\}$ which is equivalent to the boundedness of the vertical square function in~\eqref{VSF} for the semigroup. We know that, in our setting of manifolds with ends, the range of $p$ for the boundedness of $(R_p)$ and R-boundedness of the set $\{\sqrt{t}\nabla e^{-t\Delta}:t>0\}$ coincides, see \cite{HS2,BaSi}. This implies that these two conditions are equivalent in our setting. However, we do not know  any direct proof for the inverse implication, that is, from R-boundedness to $(R_p)$ condition.

\section{Preliminaries}\label{Prelim}

Throughout this article, we fix $\M=(\R^{n_1}\times\M_1)\#\cdots\#(\R^{n_l}\times\M_l)$ as a complete Riemannian manifold with $l\geq 2$ ends, unless otherwise stated. We note again that $\Delta$ is the negative Laplace-Beltrami operator on $\M$ and $\nabla$ is the gradient corresponding to the Riemannian structure. The notation $K$ will be used to denote the compact subset (also the center) of $\M$ and the sets $K_i$ and $E_i$ will be as given in Definition~\ref{ConnectedSum}. That is, $\M\setminus K$ can be expressed as the disjoint union of the ends $E_i$, where $E_i\simeq \R^{n_i}\times\M_i\setminus K_i$. 

In what follows, we always assume that each $n_i$ is at least $3$. Recall that we set $n^*=\min_{i} n_i$. \\

{\bf Notation.} For $x\in\R^n$, we define $\l<x\r>:=(1+|x|^2)^{1/2}$. We use the notation $d(x,y)$ to denote the distance between two points $x$ and $y$ in some ambient Riemannian manifold. We also employ the notation $``\lesssim"$ and $``\gtrsim"$ to denote inequalities ``up to a constant". That is, for any two quantities $a, b\in\R$, by $a\lesssim b$ we mean that there exists a constant $C>0$ such that $a\leq C\cdot b$. The same holds for $``\gtrsim"$. We denote equivalence ``up to a constant" by $a\simeq b$ which means that there exist constants $c$ and $c'$ such that $\frac{1}{c}a\leq b\leq c'a$. Lastly, for a function $g(x,y)$ of two variables, we use the notation $\nabla_xg(x,y)$ to denote the gradient with respect to the first variable.\\

We write the Stein's maximal function in the resolvent form as 
\begin{equation*}
    M_m^{res}f(x)=\sup_{t>0}|(1+t\Delta)^{-m}f(x)|,
\end{equation*}
and keep the notation $M^{res,\Delta}_m$ and $M^{res,\nabla}_m$, defined in \eqref{HMF} and \eqref{VMF}, for the resolvent based horizontal and vertical maximal operators, respectively.
With the help of the standard integral representation we get the following relation between the resolvent and the semigroup generated by $\Delta$.
\begin{equation*}
(1+t\Delta)^{-m}= \frac{1}{\Gamma(m)}\int_0^\infty e^{-s} s^{m-1}e^{-st\Delta}ds,
\end{equation*}
where $\Gamma(m)$ is the Gamma function of $m$.
This allows us to control the function $M^{res}_m f(x)$ by the Stein's maximal function $M^{exp} f(x)$. Indeed,
\begin{align}\label{Res-Exp}
\sup_{t>0}|(1+t\Delta)^{-m}f(x)| \le&  \frac{1}{\Gamma(m)}\int_0^\infty e^{-s} s^{m-1}\sup_{t>0}|e^{-st\Delta}f(x)|ds\notag\\ 
=&\sup_{t>0}|e^{-t\Delta}f(x)| \frac{1}{\Gamma(m)}\int_0^\infty e^{-s} s^{m-1}ds \lesssim   \sup_{t>0}|e^{-t\Delta}f(x)|.
\end{align}
Thus, $M^{res}_m$ is bounded on $L^p$ whenever $M^{exp}$ is.

%\subsection{Definitions}

%Note that 
%$$
%\sup_{t>0}\||\sqrt t \nabla (1+t\Delta)^{-m}|\|_{p\to p}
%\le \| M^{res}  \|_{p\to p}
%$$
%\textcolor{red}{ Hence Theorem~\ref{t1} is a consequence of the following stronger results. } 

\subsection{Higher-order Resolvents of the Laplacian on \texorpdfstring{$\M$}{e}} In this section we recall, from \cite{BaSi}, several properties of the resolvent $(\Delta+k^2)^{-m}$ for $0<k\leq1$. For the origin of these resolvent parametrix construction ideas see also \cite{CCH} and \cite{HS2}.

We first recall the following key lemma from \cite{HS2} which was crucial for the parametrix construction of the resolvent $(\Delta+k^2)^{-1}$ on $\M$. For each end $\R^{n_i}\times \mathcal{M}_i$ we choose a point $x_i^\circ$ such that $x_i^\circ\in K_i$, where $K_i$ are compact subsets of $\R^{n_i}\times\mathcal{M}_i$.  
\begin{lemma}[Key Lemma]\label{KeyLem}
    Assume that each $n_i$ is at least 3. Let $v\in C_c^\infty(\mathcal{M},\R)$. Then there exists a function $u:\mathcal{M}\times\R_+\to \R$ such that $(\Delta+k^2)^{-1}u=v$ and such that on $ith$ end we have:
    \begin{align*}
        |u(x,k)|\leq C ||v||_\infty \langle d(x_i^\circ,x)\rangle^{-(n_i-2)}\exp(-kd(x_i^\circ,x))\;\;\forall x\in\R^{n_i}\times\mathcal{M}_i,\\
        |\nabla u(x,k)|\leq C||v||_\infty \langle d(x_i^\circ,x)\rangle^{-(n_i-1)}\exp(-kd(x_i^\circ,x))\;\;\forall x\in\R^{n_i}\times\mathcal{M}_i
    \end{align*}
    for some $C>0$.
\end{lemma}
With the help of the above Lemma, it was shown in \cite{HS2} that the resolvent operator $(\Delta+k^2)^{-1}$ can be decomposed into four components. That is, for $0<k\leq1$, the resolvent can be written as
\begin{equation*}
    (\Delta+k^2)^{-1}=\sum_{j=1}^4 G_j(k).
\end{equation*}
We recall below the definition of each of these components.
Let $\phi_i\in C^\infty(\M)$ be a function supported entirely in $\R^{n_i}\times\M_i\setminus K_i$ that is identically equal to 1 everywhere outside of a compact set. Define $v_i:=-\Delta \phi_i$ and note that $v_i$ is compactly supported. Let $u_i$ be the function for $v_i$ whose existence is asserted by Lemma~\ref{KeyLem}. The $G_1(k)$ term is entirely supported on the diagonal ends and has the following kernel
\begin{equation*}
    G_{1}(k)(x,y) := \sum_{i = 1}^{l} {(\Delta_{\R^{n_{i}}\times
		\mathcal{M}_{i}} + k^{2})}^{-1}(x,y) \phi_{i}(x) \phi_{i}(y).
\end{equation*}
 Let $G_{int}(k)$ be an interior parametrix for the resolvent that is supported close to the compact subset
 \begin{equation*}
     K_{\Delta}:=\{(x,x): x\in K\}\subset \M^2
 \end{equation*}
and agrees with the resolvent of $\Delta_{\R^{n_i}\times\M_i}$ in a small neighbourhood of $K_\Delta$ intersected with the support of $\nabla\phi_i(x)\phi_i(y)$. This gives us the $G_2(k)$ term whose kernel has the following representation
$$
G_{2}(k)(x,y) := \bigg(1 - \sum_{i = 1}^{l} \phi_{i}(x) \phi_{i}(y)\bigg)G_{int}(k)(x,y) .
$$
The $G_{3}(k)$ term has the nice property that its kernel is multiplicatively separable into functions of $x$ and $y$. That is,
$$
G_{3}(k)(x,y) := \sum_{i = 1}^{l} (\Delta_{\R^{n_{i}}\times
		\mathcal{M}_{i}} + k^{2})^{-1}(x_{i}^{\circ},y) u_{i}(x,k) \phi_{i}(y),
$$
where $(\Delta + k^{2})u_i=-\Delta \phi_i =\nu_i $.
The final term $G_4(k)$ is a correction term, and for this we first define the error term $E(k)$ as 
$$
(\Delta + k^{2})(G_{1}(k) + G_{2}(k) + G_{3}(k)) = I + E(k).
$$
The operator $G_4(k)$ is given by 
$$
G_{4}(k)(x,y) := - (\Delta + k^{2})^{-1} v_{y}(x),
$$
where $v_{y}(x) := E(k)(x,y)$ and is compactly supported. The error term was computed in \cite{HS2} and has the following representation
\begin{equation*}
    E(k)=\sum_{i=1}^l(E_1^i(k)+E_2^i(k))+E_3(k).
\end{equation*}
Here
$$
E_1^i(k)(x,y):=-2\nabla\phi_i(x)\phi_i(y)\l[\nabla_x(\Delta_{\R^{n_i}\times\M_i}+k^2)^{-1}(x,y)-\nabla_xG_{int}(k)(x,y)\r],
$$
$$
E_2^i(k)(x,y):=\phi_i(y)v_i(x)\l(-(\Delta_{\R^{n_i}\times\M_i}+k^2)^{-1}(x,y)+G_{int}(k)(x,y)+(\Delta_{\R^{n_i}\times\M_i}+k^2)^{-1}(x_i^{\circ},y)\r)
$$
for $i=1,...,l$, and
$$
E_3(k)(x,y):=\l((\Delta+k^2)G_{int}(k)(x,y)-\delta_y(x)\r)\l(1-\sum_{i=1}^l\phi_i(x)\phi_i(y)\r),
$$
where $\delta_y$ is the Dirac-delta function centered at $y$.\\

Since we are interested in the higher-order resolvents on $\M$, by the above decomposition we can write $(\Delta+k^2)^{-m}$ as
\begin{equation*}
    (\Delta+k^2)^{-m}=\frac{(-1)^{m-1}}{(m-1)!}\partial_{k^2}^{(m-1)}(\Delta+k^2)^{-1}=\sum_{j=1}^4 H_j^{(m)}(k),
\end{equation*}
where $H^{(m)}_j(k):=\frac{(-1)^{m-1}}{(m-1)!}\partial_{k^2}^{(m-1)}G_j(k)$. For simplicity, we use the shorthand notation $H_j(k)$.

Now, following \cite{HS2}, for $a=1,2$ we define weight functions $\omega_a : \mathcal{M} \times [0, 1] \to (0,\infty)$ by 
\begin{equation}\label{DefOmega}
\omega_a(x,k) = \begin{cases}  \hspace{2cm}   1, &\quad x \in K, \\
\langle d(x_i^\circ,x)\rangle^{-(n_i -a)} \exp(-{ckd(x_i^\circ,x)}), &\quad x\in\R^{n_i}\times\M_i\setminus K_i,\;1\leq i\leq l.
\end{cases}
\end{equation}
We recall below the propositions from \cite{BaSi} that estimates the terms $H_3(k)$ and $H_4(k)$ using these weight functions.

\begin{proposition}\cite[Proposition 3.4]{BaSi}\label{PropH3}
    The kernel of the operator $H_3(k)$ satisfies
    \begin{equation*}\label{H3a}
    |H_3(k)(x,y)|
\lesssim k^{-2(m-1)}\omega_2(x,k)\omega_2(y,k)
\end{equation*}
and
\begin{equation*}\label{H3b}
    |\nabla_x H_3(k)(x,y)|
\lesssim k^{-2(m-1)}\omega_1(x,k)\omega_2(y,k)
\end{equation*}
for all $x,y\in\M$, and $0<k\leq1$.
\end{proposition}

\begin{proposition}\cite[Proposition 3.5]{BaSi}\label{PropH4}
 The kernel of the operator $H_4(k)$ satisfies
    \begin{equation*}\label{H4a}
    |H_4(k)(x,y)|
\lesssim k^{-2(m-1)}\omega_2(x,k)\omega_1(y,k)
\end{equation*}
and
\begin{equation*}\label{H4b}
    |\nabla_x H_4(k)(x,y)|
\lesssim k^{-2(m-1)}\omega_1(x,k)\omega_1(y,k)
\end{equation*}
for all $x,y\in\M$, and $0<k\leq1$.   
\end{proposition}

%we rewrite the higher-order analogues of the key lemma \ref{KeyLem}, whose proof is mentioned in \cite[Proposition 3.1]{BaSi}.

%\begin{lemma}
 %   Let $v$ and $u$ be as defined in Lemma~\ref{KeyLem} such that $(\Delta+k^2)u=v$, and satisfies
  %  \begin{align*}
   %     |u(x,k)|\lesssim ||v||_\infty\omega_2(x,k),\;\text{ and }   \;\; |\nabla u(x,k)| \lesssim ||v||_\infty \omega_1(x,k)
   % \end{align*} for all $x\in\M$ and $0\leq k\leq 1$. For $j\in\mathbb{N}$, define
   % \begin{equation*}
    %    u^{(j)}:=\partial_{k^2}^{(j)}u=(-1)^{j}j!(\Delta+k^2)^{-(j+1)}v.
   % \end{equation*} Then,
    %\begin{equation*}
     %   |u^{(j)}(x,k)|\lesssim k^{-2j}||v||_\infty\omega_2(x,k),
    %\end{equation*} for all $x\in\M$ and $0<k\leq1$.
%\end{lemma}

Let $\mathcal{N}$ be a compact Riemannian manifold. Then $\R^n\times\mathcal{N}$ is a manifold that satisfies the doubling property. On this manifold, we show in the following lemma that the vertical maximal operator is bounded on $L^p$ for $1<p\leq\infty$ and satisfies weak type (1,1) estimates. %It is worth noticing that we prove the result here for the maximal operator involving the semigroup generated by $\Delta_{\R^n\times\M}$. 

\begin{lemma}\label{MaxLemma}
%Let $\mathcal{N}$ be a compact Riemannian manifold. 
The vertical maximal operator $M^{exp,\nabla}f(x)=\sup_{t>0}|\sqrt{t}\nabla e^{-t\Delta_{\R^n\times \mathcal{N}}}f(x)|$ 
   is bounded on $L^p(\R^n\times \mathcal{N})$ for $1<p\leq\infty$ and is weak-type $(1,1)$.
   The same statement holds for $M^{res,\nabla}_m=\sup_{t>0}|\sqrt{t}\nabla ({1+t\Delta_{\R^n\times \mathcal{N}}})^{-m}f(x)|$.
\end{lemma}

\begin{proof}
    The heat kernel on $\R^n\times\mathcal{N}$ has the form 
    $$    e^{-t\Delta_{\R^n\times \mathcal{N}}}(x,y) \lesssim \left(\frac{1}{ t^{n/2}}+\frac{1}{ t^{N/2}}\right)\exp\left(-\frac{d(x,y)^2}{4t}\right),$$
    where $N= \text{dim } \mathcal{N}+n$, and satisfies the following spatial derivative estimate
    \begin{equation*}
        |\nabla e^{-t\Delta_{\R^n\times\mathcal{N}}}(x,y)|\lesssim t^{-1/2} \left(\frac{1}{ t^{n/2}}+\frac{1}{ t^{N/2}}\right)\exp\left(-\frac{d(x,y)^2}{4t}\right).
    \end{equation*}
    On controlling the right hand side by an infinite sum we obtain
   \begin{equation*}
        |\sqrt{t}\nabla e^{-t\Delta_{\R^n\times\mathcal{N}}}(x,y)|\lesssim  \sum_{k=1}^\infty \frac{k^{-N/2}e^{-k}}{B(x,\sqrt{kt})}\chi_{B(x,\sqrt{kt})}.
    \end{equation*} 
    Thus,
    \begin{align*}
        M^{exp,\nabla}f(x)&=\sup_{t>0}|\sqrt{t}\nabla e^{-t\Delta_{\R^n\times \mathcal{N}}}f(x)|
        =\sup_{t>0}\l|\sqrt{t}\nabla \int_{\R^n\times\mathcal{N}}e^{-t\Delta_{\R^n\times \mathcal{N}}}(x,y)f(y)dy\r|\\ &\lesssim\sup_{t>0}\l|\int_{\R^n\times\M}\sum_{k=1}^\infty\frac{k^{-N/2}e^{-k}}{B(x,\sqrt{kt})}\chi_{B(x,\sqrt{kt})}\r|\\
        &\lesssim\sup_{t>0}\l|\sum_{k=1}^\infty k^{-N/2}e^{-k}\frac{1}{B(x,\sqrt{kt})}\int_{B(x,\sqrt{kt})}f(y)dy\r|\\
        &\lesssim Mf(x),\;\;\;\; \text{since} \sum_{k=1}^\infty k^{-N/2}e^{-k}<\infty.
    \end{align*}
    Here $Mf$ is the Hardy-Littlewood maximal function which is bounded on $L^p(\R^n\times \mathcal{N})$ for $1<p\leq\infty$ and is weak-type (1,1). This implies that $M^{exp,\nabla}f$, and hence $M^{res,\nabla}_mf$, are weak-type (1,1) and bounded on $L^p(\R^n\times\mathcal{N})$ for all $1<p\leq\infty$.
\end{proof}

\begin{proposition}\label{FS}
	Consider again the vertical maximal operator
 $M^{exp,\nabla}$ on $L^p(\R^n\times\mathcal{N})$. Then, for any sequence $(f_j)_{j=1}^\infty\in L^p(\R^n\times\mathcal{N})$
\begin{equation*}
\l|\l|\Bigg(\sum_{j=1}^\infty|M^{exp,\nabla}f_j|^2\Bigg)^{1/2}\r|\r|_{1,\infty}
\leq C\l|\l|\Bigg(\sum_{j=1}^\infty|f_j|^2\Bigg)^{1/2}\r|\r|_1
\end{equation*}
for some constant $C>0$ and
\begin{align*}
\l|\l|\Bigg(\sum_{j=1}^\infty|M^{exp,\nabla}f_j|^2\Bigg)^{1/2}\r|\r|_p
\leq A_{n,p,\mathcal{N}}\l|\l|\Bigg(\sum_{j=1}^\infty|f_j|^2\Bigg)^{1/2}\r|\r|_p
\end{align*}
for all $1<p<\infty$ and some constant $A_{n,p,\mathcal{N}}$ (depending on $p$ and the manifold $\R^n\times\mathcal{N}$). 
\end{proposition}

\begin{proof}
Using the similar idea as in the proof of Lemma~\ref{MaxLemma} we obtain that $M^{exp,\nabla}$ is controlled by the Hardy-Littlewood maximal operator $M$. This implies that for a sequence of functions $(f_j)_{j=1}^\infty$ in $L^p(\R^n\times\mathcal{N})$ we have
\begin{align*}
 \bigg|\bigg|\bigg(\sum_{j=1}^\infty|M^{exp,\nabla}f_j|^2\bigg)^{1/2}\bigg|\bigg|_p
 &\lesssim\bigg|\bigg|\bigg(\sum_{j=1}^\infty|Mf_j|^2\bigg)^{1/2}\bigg|\bigg|_p
\leq A_{n,p,\mathcal{N}}\bigg|\bigg|\bigg(\sum_{j=1}^\infty|f_j|^2\bigg)^{1/2}\bigg|\bigg|_p,
\end{align*}
where the last inequality follows from \cite[Theorem~1.2]{GLY}, see also \cite[Theorem~1, p. 51]{St}, and holds for all $1<p<\infty$ and is also weak-type~(1,1).
\end{proof}

The following lemma for bounded linear operators on a Banach space $X$ is crucial to prove the vector-valued version of the vertical maximal function. The statement is quoted from \cite[Section 2.5, Chapter X, p. 450]{St} and we include its proof here for completeness.
\begin{lemma}\label{FSLem}
    Suppose $S$ is a bounded linear operator on $L^p(X,d\mu)$, where $X$ is a Banach space. Let $f=(f_j)_{j=1}^\infty$ be a sequence in $L^p(X,d\mu)$. Then we have the following relation
    \begin{equation}\label{FSEqnGen}
    \l|\l|\Big(\sum_{j=1}^\infty |Sf_j|^2\Big)^{1/2}\r|\r|_p\leq ||S||_{p\to p}\l|\l|\Big(\sum_{j=1}^\infty |f_j|^2\Big)^{1/2}\r|\r|_p.
    \end{equation}

    %$f=(f_1,...,f_j)$ be any $j$-tuple of functions in $L^p(X,d\mu)$ then $Sf=(Sf_1,...,Sf_j)$. F

\end{lemma}

\begin{proof}
To prove the above relation it is enough to show that for any $J\in\mathbb{N}$ we have
\begin{equation}
    \l|\l|\Big(\sum_{j=1}^{J} |Sf_j|^2\Big)^{1/2}\r|\r|_p\leq ||S||_{p\to p}\l|\l|\Big(\sum_{j=1}^{J} |f_j|^2\Big)^{1/2}\r|\r|_p.
    \end{equation}
For a unit vector $\o=(\o_{J})\in\mathbb{C}^{J}$ define $f_\o:=\l<f,\o\r>$ such that $Sf_\o=\l<Sf,\o\r>$, where $Sf=(Sf_j)_{j=1}^{J}$.
Let $\S_{J}$ be the boundary of the sphere $\mathbb{S}^{J}$ in $\mathbb{C}^{J}$. 
    Consider \begin{align*}
        \int_{\S_{J}}\l(\int_X |Sf_\o(x)|^pdx\r)d\o&=\int_{\S_{J}}\l(\int_X |\l<Sf(x),\o\r>|^pdx\r)d\o\\
        &=\int_{\S_{J}}\l(\int_X |Sf(x)|^p\l|\l<\frac{Sf(x)}{|Sf(x)|},\o\r>\r|^pdx\r)d\o\\
        &=\int_X |Sf(x)|^p\l(\int_{\S_{J}}\l|\l<\frac{Sf(x)}{|Sf(x)|},\o\r>\r|^pd\o\r) dx=A_{J,p}\int_X |Sf(x)|^p dx,
    \end{align*}
    where $A_{J,p}$ is a positive constant that depends on $J$ and $p$.
    We also have
    \begin{align*}
         \int_{\S_{J}}\l(\int_X |Sf_\o(x)|^pdx\r)d\o&\leq  ||S||_{p\to p}^p\int_{\S_{J}}\l(\int_X |f_\o(x)|^pdx\r)d\o\\
         &=||S||^p_{p\to p}\int_{\S_{J}}\l(\int_X |f(x)|^p\l|\l<\frac{f(x)}{|f(x)|},\o\r>\r|^pdx\r)d\o\\
        &=||S||^p_{p\to p}\int_X |f(x)|^p\l(\int_{\S_{J}}\l|\l<\frac{f(x)}{|f(x)|},\o\r>\r|^pd\o\r) dx\\
        &=A_{J,p}||S||^p_{p\to p}\int_X |f(x)|^p dx.
    \end{align*}
    Since the above two relations hold independent of $J$, we obtain \eqref{FSEqnGen}.
\end{proof}

\section{Proofs of Main Theorems}
%\com{AS: describe all assumptions in Thm 3.1}

%\com{AS: Note that the $G_p$ is not open ended here.}

 \subsection{Proof of Theorem~\ref{GpThm}}

	First note that 
	\begin{equation}\label{www}
	\|\;|{\sqrt t} \nabla (1+t\Delta)^{-m}|\;\|_{p \to p}\sim  {C} \quad \forall{0<t\leq1}.
	\end{equation}
	The proof of the above estimate is a consequence of ideas discussed  in Sections 5 and 7  of \cite{HS2}. The argument is a  modification of the proof of \cite[Proposition 5.1]{HS2}.  The similar argument is also described in \cite[Proposition 5.1]{BaSi}, see also \cite[Proposition 2.4]{HS2}.
	
	\bigskip
Now, it follows from \eqref{www} that %we can only 
it is enough to consider $t>1$.
	We substitute $t=1/k^2$ and in what follows we consider only $k<1$. 
Then, using the parametrix for the resolvent $(k^2+\Delta)^{-m}$, we obtain
\begin{align*}
    \|\;|{\sqrt t} \nabla (1+t\Delta)^{-m}|\;\|_{p \to p}=\|\;|k^{2m-1}\nabla(k^2+\Delta)^{-m}|\;\|_{p\to p}= \Big|\Big|\;\Big|k^{2m-1} \nabla \sum_{j=1}^4H_j(k)\Big|\;\Big|\Big|_{p \to p}.
\end{align*}
To obtain the $L^p$ bounds, we shall analyze all factors $H_j(k)$ for $j=1,\ldots, 4$. \\

$\bullet$ The $H_1(k)$ term. \\

From the definition of $H_1(k)$, we need to analyze the boundedness of 
\begin{equation}
k^{2m-1} \nabla \Big( (\Delta_{\R^{n_i} \times \mathcal{M}_i} + k^2)^{-m}(x, y) \phi_i(x) \phi_i(y) \Big). 
\end{equation}
The kernel of $H_1(k)$ can be viewed as an operator acting on $\R^{n_i} \times \mathcal{M}_i$.
 We split  this kernel into two pieces, according to whether the derivative $\nabla$ hits the function $\phi(x)$ or the resolvent factor.
 \medskip
 
Case (1): $H_{1,a}$--- when the derivative hits the resolvent. That is, we need to compute the $L^p$ norm of 
 \begin{equation}\label{G1Term}
     \phi_i k^{2m-1} \nabla(\Delta_{\R^{n_i} \times \mathcal{M}_i} + k^2)^{-m}   \phi_i.
 \end{equation} 
Let $R_{1,k}(x,y)= \phi_i(x)k^{2m-1}\nabla(\Delta_{\R^{n_i}\times\mathcal{M}_i}+k^2)^{-m}(x,y)\phi_i(y)$ be the kernel of the above operator. 
Define the norm $$||T(x,y)||_{L^1_xL^\infty_y}:=\sup_{y\in\M}\int_\M |T(x,y)|dx$$ for some operator $T$.
If we show that
 \begin{equation*}
     ||R_{1,k}(x,y)||_{L^1_xL^\infty_y}\leq C
 \end{equation*}
 and
 \begin{equation*}
     ||R_{1,k}(x,y)||_{L^\infty_xL^1_y}\leq C
 \end{equation*}
 for some constant $C>0$, then the operator in \eqref{G1Term} will be bounded on $L^p$ for $p=1$ and $p=\infty$, and hence for all $1\leq p\leq\infty$, by interpolation.
 Indeed, consider 
 \begin{align*}
     ||R_{1,k}(x,y)||_{L^1_xL^\infty_y}&=\sup_{y\in\mathcal{M}}\int_{\mathcal{M}}|R_{1,k}(x,y)|dx\\
     &=\sup_{y\in\R^{n_i}\times\mathcal{M}_i}\int_{\R^{n_i}\times\M_i}| \phi_i(x)k^{2m-1}\nabla(\Delta_{\R^{n_i}\times\mathcal{M}_i}+k^2)^{-m}(x,y)\phi_i(y)|dx\\
&\leq\sup_{y\in\R^{n_i}\times\mathcal{M}_i\setminus K_i}\int_{\R^{n_i}\times\M_i\setminus K_i}\,| k(d(x,y)^{-(N-1)}+d(x,y)^{-(n_i-1)})e^{-ckd(x,y)}|dx\\
     &\hspace{5cm}\text{ (follows by \cite[Proposition 2.2]{BaSi})}\\
     &\lesssim\int_0^\infty kr^{-(n_i-1)}e^{-ckr} r^{n_i-1}dr\leq C<\infty.
 \end{align*}
 The last relation follows by a change of variable, and the second last relation follows from the fact that $\phi_i$ takes the value 1 outside a compact subset $K_i$ of $\R^{n_i}\times\mathcal{M}_i$ and is zero elsewhere. Since the kernel is symmetric in $x$ and $y$ variables, we have $||R_{1,k}(x,y)||_{L^\infty_xL^1_y}\leq C$.
  \medskip
 
 Case (2): $H_{1,b}$--- when the derivative hits the function $\phi(x)$. Here we shall show that
 \begin{eqnarray*}
 	\|\nabla \phi_i k^{2m-1} (\Delta_{\R^{n_i} \times \mathcal{M}_i} + k^2)^{-m}\|_{p\to p } \le C < \infty
 \end{eqnarray*}
 for all  $1\leq p \le  n_i$.
 
 Set $D=\{(x,y)\in\mathcal{M}^2: d(x,y)\leq1\}$ and let $\chi_{D}$ be the characteristic function of the set $D$. By the standard argument of \cite[Proposition 2.2]{BaSi}, we get that
 \begin{equation*}
||\chi_Dk^{2m-1}(\Delta_{\R^{n_i}\times\mathcal{M}_i}+k^2)^{-m}||_{p\to p}\lesssim 1.
 \end{equation*}
 Indeed,
\begin{equation*}
|\chi_Dk^{2m-1}(\Delta_{\R^{n_i}\times\mathcal{M}_i}+k^2)^{-m}(x,y)| \le Cd(x,y)^{2-N}e^{-ckd(x,y)}
 \end{equation*}
 for some constant $C$ and for all $x,y \in \R^{n_i}\times\mathcal{M}_i$  independent of $k$. The operator now is bounded because it is estimated by a ``convolution" with $L^1(\R^{n_i}\times\mathcal{M}_i)$ function.

 Consider now the kernel $R_{2,k}(x,y)$ of the operator $(1-\chi_D)k^{2m-1}(\Delta_{\R^{n_i}\times\mathcal{M}_i}+k^2)^{-m}$. Then, again using \cite[Proposition 2.2]{BaSi} and following as in the $H_{1,a}$ term, we obtain
\begin{align*}
    ||R_{2,k}(x,y)||_{L^\infty_xL^q_y}&=\sup_{x\in\mathcal{M}}\bigg(\int_\mathcal{M}|R_{2,k}(x,y)|^qdy\bigg)^{1/q}\\
   &\lesssim\sup_{x\in\mathcal{M}}\bigg(\int_\mathcal{M}|kd(x,y)^{2-n_i}e^{-ckd(x,y)}|^qdy\bigg)^{1/q}\\
   &\simeq  \l(\int_{1}^\infty\left|kr^{2-n_i}e^{-ckr}\right|^{q}r^{n_i-1}dr\r)^{1/q}\\
   &\simeq \l(\int_k^\infty \frac{k^qr^{(2-n_i)q+n_i-1}}{k^{(2-n_i)q+n_i}}e^{-r}dr\r)^{1/q}
   \lesssim k^{-1+n_i(1-1/q)}\Big(\int_k^\infty r^{(2-n_i)q+n_i-1}e^{-r}dr\Big)^{1/q}\\
   &\lesssim k^{-1+n_i(1-1/q)}.
\end{align*}
Similarly,
\begin{equation*}
     ||R_{2,k}(x,y)||_{L^q_xL^\infty_y}\lesssim k^{-1+n_i(1-1/q)}.
\end{equation*}
This implies that this operator is bounded as a map from $L^{q'}\to L^\infty$ and from $L^1\to L^q$ with operator norm bounded by $k^{-1+n_i(1-1/q)}$. Interpolating, we find that
\begin{equation*}
    ||(1-\chi_D)k^{2m-1}(\Delta_{\R^{n_i}\times\mathcal{M}_i}+k^2)^{-m}||_{p\to q}\lesssim k^{-1+n_i(1/p-1/q)}
\end{equation*}
for all $p<q$. Hence,
\begin{align*}
    &||\nabla\phi_i(1-\chi_D)k^{2m-1}(\Delta_{\R^{n_i}\times\mathcal{M}_i}+k^2)^{-m}||_{p\to p}\\
    &\lesssim||(1-\chi_D)k^{2m-1}(\Delta_{\R^{n_i}\times\mathcal{M}_i}+k^2)^{-m}||_{p\to q}\\
    &\lesssim k^{-1+n_i(1/p-1/q)}\lesssim 1
\end{align*}
for all $p<q$ such that $1/p-1/q\geq1/n_i$. This, and the fact that $\nabla \phi_i$ is compactly supported, implies that 
\begin{equation*}
    ||\nabla\phi_ik^{2m-1}(\Delta_{\R^{n_i}\times\mathcal{M}_i}+k^2)^{-m}||_{p\to p}\lesssim 1
\end{equation*}
for all $p\leq n_i$.\\
 %Maybe a bit technical but should not be too difficult see \cite{HS2}
 
 %*****************H2 Term*********************************************************************
 
 $\bullet$ The $H_2(k)$ term. \\
 
 From the definition of $G_2(k)$, which occurs in the term $H_2(k)$, it can be noticed that the operators $\{\nabla H_2(k)\}_{k\in(0,1)}$ form a family of pseudodifferential operators of order $1-2m$ with the Schwartz kernel having compact support and depending smoothly on $k$. 
 At the scale of an individual chart on $\M$, it is clear that in local coordinates, the symbol of $\nabla H_2(k)$, denoted by $a_2(k)(x,\xi)$, will satisfy
 \begin{eqnarray*}
     |\partial_\xi^\alpha a_2(k)(x,\xi)|\lesssim (|\xi^2|+k^2)^{\frac{1-2m-|\alpha|}{2}}
 \end{eqnarray*}
 for all multi-indices $\alpha\geq0$ and $k\in(0,1)$. Thus, from the standard pseudodifferential operator theory, see for instance \cite[Section 0.2]{Taylor}, we get
 
\begin{equation*}
    |\nabla_x H_2(k)(x,y)|\lesssim k^{N+1-2m-b}d(x,y)^{-b}
\end{equation*}
for any $b>N+1-2m$. Setting $b=N-\frac{1}{2}$ then gives
\begin{equation*}
    |k^{2m-1}\nabla_x H_2(k)(x,y)|\lesssim k^{\frac{1}{2}} d(x,y)^{\frac{1}{2}-N}.
\end{equation*}
 Thus, the operator with the $H_2(k)$ term has kernel that is compactly supported and is controlled by $d(x,y)^{\frac{1}{2}-N}$. Therefore, we obtain the boundedness here for all $1\le p \le \infty$. \\

 %*****************H3 Term*********************************************************************
 
 $\bullet$ The $H_3(k)$ term.\\

 To obtain the boundedness in this case, we consider the kernel $H_3(k)(x,y)$. Now, Proposition~\ref{PropH3} gives
 \begin{align*}
     |H_3(k)(x,y)|\lesssim k^{-2(m-1)}\omega_2(x,k)\omega_2(y,k)
 \end{align*}
and
 \begin{equation*}
     |\nabla_x H_3(k)(x,y)|\lesssim k^{-2(m-1)}\omega_1(x,k)\omega_2(y,k),
 \end{equation*}
 where $\omega_a(x,k)$ is defined in \eqref{DefOmega} for $a=1,2$.
 \begin{align*}
\text{Consider }\hspace{1cm} \| k^{2m-1} \nabla H_3(k)f\|_{p}&=\l(\int_{\M}|k^{2m-1}\nabla H_3(k)f(x)|^p dx\r)^{1/p}\\
    &=\l(\int_{\M}\left|k^{2m-1}\int_{\M} \nabla_x H_3(k)(x,y)f(y) dy\right|^p dx\r)^{1/p}\\
    &\lesssim \l(\int_{\M}\left|k\int_{\M} \omega_1(x,k)\omega_2(y,k) f(y) dy\right|^p dx\r)^{1/p}\\
    &\simeq k ||f||_p\l(\int_\M|\omega_1(x,k)|^pdx\r)^{1/p} \l(\int_\M|\omega_2(y,k)|^{p'}dy\r)^{1/p'}.
\end{align*}
Now, depending on whether $x$ and $y$ are in the compact set or outside, we have the following outcomes from the definition of $\omega_a$
\begin{align*} \l(\int_\M|\omega_1(x,k)|^pdx\r)^{1/p}&= \l(\int_{ K_i}1 dx\r)^{1/p}+\l(\int_{\R^{n_i}\times\M_i\setminus K_i}|\omega_1(x,k)|^pdx\r)^{1/p}\\
&\simeq \l(\int_{\R^{n_i}\times\M_i\setminus K_i}\l|\left<d(x_i^\circ,x)\right>^{-(n_i-1)}e^{-ck(d(x_i^\circ,x)}\r|^{p}dx\r)^{1/p}\\
&\simeq  \l(\int_{1}^\infty\left|r^{-(n_i-1)}e^{-ckr}\right|^{p}r^{n_i-1}dr\r)^{1/p}\\
&\simeq  \l(\int_{1}^\infty r^{-(n_i-1)p+n_i-1}e^{-ckr}dr\r)^{1/p}\\
&\simeq\l(\int_k^\infty \frac{r^{-(n_i-1)p+n_i-1}}{k^{-(n_i-1)p+n_i}}e^{-r}dr\r)^{1/p}\\
&\simeq \l(k^{-\alpha}\int_k^\infty r^{\alpha-1}e^{-r}dr\r)^{1/p},\;\text{ where } \alpha=-(n_i-1)p+n_i.
\end{align*}
By simple computations we obtain
\begin{align}
\left(k^{-\alpha}\int_k^\infty r^{\alpha-1}e^{-r}dr\right)^{1/p}\lesssim \begin{cases}
    k^{(\frac{n_i}{p'}-1)}, &\alpha>0, i.e., p<\frac{n_i-1}{n_i},\\
    1,                       &\alpha<0, i.e., p>\frac{n_i-1}{n_i}.
    \end{cases}
\end{align}
Also,
\begin{align*}
 \l(\int_\M|\omega_2(y,k)|^{p'}dy\r)^{1/p'}&= \l(\int_{ K_j}1 dy\r)^{1/p}+    \l(\int_{\R^{n_j}\times \M_j\setminus K_j}|\omega_2(y,k)|^{p'}dy\r)^{1/p'}\\
 &\simeq \l(\int_{\R^{n_j}\times\M_j\setminus K_j}\l|\left<d(x_j^\circ,x)\right>^{-(n_j-2)}e^{-ck(d(x_j^\circ,x)}\r|^{p'}dx\r)^{1/p'}\\
 &\simeq  \l(\int_{1}^\infty\left|r^{-(n_j-2)}e^{-ckr}\right|^{p'}r^{n_j-1}dr\r)^{1/p'}\\
 &\simeq  \l(\int_{1}^\infty r^{-(n_j-2)p'+n_j-1}e^{-ckr}dr\r)^{1/p'}\\
 &\simeq \l(\int_k^\infty \frac{r^{-(n_j-2)p'+n_j-1}}{k^{-(n_j-2)p'+n_j}}e^{-r}dr\r)^{1/p'}\\
 &\simeq \l(k^{-\beta}\int_k^\infty r^{\beta-1}e^{-r}dr\r)^{1/p'}, \text{ where }\beta=-(n_j-2)p'+n_j.
\end{align*}
Set 
 \begin{align*}
\gamma_\beta(k)= \begin{cases}
    k^{(\frac{n_j}{p}-2)}, &\beta>0, i.e., p>\frac{n_j}{2},\\
    1,                       &\beta<0, i.e., p<\frac{n_j}{2}.
    \end{cases}
\end{align*}
This gives
 \begin{align}
\left(k^{-\beta}\int_k^\infty r^{\beta-1}e^{-r}dr\right)^{1/p'}\simeq
\gamma_\beta(k).
\end{align}
 We thus have
\begin{align*}
 \|k^{2m-1}\nabla H_{3}(k)\|_{p \to p} \lesssim \max_{n_i,n_j}
   k&\l(k^{-\alpha}\int_k^\infty r^{\alpha-1}e^{-r}dr\r)^{1/p}\left( k^{-\beta}\int_{k}^\infty s^{\beta-1}e^{-s}ds\right)^{1/p'}    ,
\end{align*}
 where
\begin{align}\label{Eq_G3}\nonumber
   k&\l(k^{-\alpha}\int_k^\infty r^{\alpha-1}e^{-r}dr\r)^{1/p}\left( k^{-\beta}\int_{k}^\infty s^{\beta-1}e^{-s}ds\right)^{1/p'}\\&\lesssim k 
   \begin{cases}
     k^{(\frac{n_i}{p'}-1)+(\frac{n_j}{p}-2)}, &\alpha>0,\beta>0 \implies\frac{n_j}{2}<p<\frac{n_i}{n_i-1},\\
    k^{(\frac{n_j}{p}-2)},              &\alpha<0,\beta>0\implies p>\frac{n_i}{n_i-1}\text{ and }p>\frac{n_j}{2},\\
    k^{(\frac{n_i}{p'}-1)},       &\alpha>0,\beta<0\implies p<\frac{n_i}{n_i-1}\text{ and  }p<\frac{n_j}{2} ,     \\
    1,                         & \alpha<0,\beta<0 \implies \frac{n_i}{n_i-1}<p<\frac{n_j}{2}.
    \end{cases}
\end{align}
Since $\frac{n_j}{2}\geq\frac{3}{2}$ for all $j$, and $\frac{n_i}{n_i-1}\leq \frac{3}{2}$ for all $i$, so the first case cannot happen, and from the remaining cases we obtain boundedness for all $1\leq p\leq n_j$. This is because the exponent $\frac{n_j}{p}-1$ of k in the second case is non-negative only if $p\leq n_j$.\\

 %*****************H4 Term********************************************************************* 
 
 $\bullet$ The $ H_4(k)$ term.\\
 
 This term can be treated in the same way as the $H_3(k)$ term, with the difference that it vanishes to an additional order in the right (primed) variable.
 Indeed, by Proposition~\ref{PropH4}, we have
\begin{align}\label{G4Cond}
|\nabla_x H_4(k)(x,y)|&\lesssim k^{-2(m-1)}\omega_1(x,k)\omega_1(y,k)\notag.
\end{align}
This implies
\begin{align*}
    ||k^{2m-1}\nabla H_4(k)f||_{p}&=\l(\int_{\M}|k^{2m-1}\nabla H_4(k)f(x)|^pdx\r)^{1/p}\\
    &=\l(\int_\M\left|\int_\M k^{2m-1}\nabla_x H_4(k)(x,y)f(y)dy\right|^p dx\r)^{1/p}\\
    &\;\;\;\times\l(\int_{\M}\left|k\int_{\M} \omega_1(x,k)\omega_1(y,k) f(y) dy\right|^p dx\r)^{1/p}\\
    &\simeq k ||f||_p\l(\int_\M|\omega_1(x,k)|^pdx\r)^{1/p} \l(\int_\M|\omega_1(y,k)|^{p'}dy\r)^{1/p'}.
\end{align*}
Following on similar lines as in the $H_3$ term, we obtain boundedness for all $1\leq p\leq\infty.$\\

%\newpage

Now, taking $H_j(k)$ terms together for $j=1,...,4$, we obtain $||\;|\sqrt{t}\nabla(1+t\Delta)^{-m}|\;||_{p\to p}\lesssim1$ for all $t>0$ and all $1\leq p\leq n^*$, where $n^*=\min_{i}\{n_i\}$. Also, from \eqref{Eq_G3} we obtain $||\;|\sqrt{t}\nabla(1+t\Delta)^{-m}|\;||_{p\to p}\lesssim (\sqrt{t})^{1-\frac{n^*}{p}}$ for $p\in [n^*,\infty]$ and $t\geq1$.\\

We shall now show that for $m=1$ we also have
\begin{equation}\label{a20}
	\|\;|{\sqrt t} \nabla (1+t\Delta)^{-1}|\;\|_{p \to p} \ge  {c}  (\sqrt t)^{1-n^*/p}   \quad \forall{t\geq1}
	\end{equation}
and for all $p\in [n^*, \infty]$, where $c>0$ is some constant. %This would imply unboundedness of the resolvent operator for $p\ge n^*$.\\
It can be noticed that in the previous part of the proof the only parts that did not give the uniform boundedness for the entire range of $p\in[1,\infty]$ were $H_1(k)$ and $H_3(k)$. Thus, to obtain \eqref{a20}, % unboundedness, 
we only have to look for these two cases. Since we are working with $m=1$, we only need to consider the cases $G_1(k)$ and $G_3(k)$. Therefore, we %should
estimate the following expression 
%Therefore,
\begin{align*}
\|\;| k\nabla (G_1(k)+G_3(k))f(x)|\;\|_{p\to p}
&=  \l|\l|\;\l|k \int_\M \nabla[G_1(k)(x,y)+G_3(k)(x,y)]f(y)dy\r|\;\r|\r|_{p\to p}.
\end{align*}

Recalling from the previous part of the proof, the $H_1(k)$ term were divided into two parts, $H_{1,a}$ and $H_{1,b}$ depending on whether the gradient hits the resolvent factor or the function $\phi_i$. The term where the gradient hits the resolvent, $H_{1,a}$, was proved to be uniformly bounded on $L^p$ for all $p\in[1,\infty]$. It is, therefore, sufficient to verify the lower bounds from  \eqref{a20} for the operator corresponding to the following kernel 
\begin{align}\label{G_1+G_3}
    &k\sum_{i = 1}^{l}\int_\M \l|\nabla \phi_{i}(x){(\Delta_{\R^{n_{i}}\times
		\mathcal{M}_{i}} + k^{2})}^{-1}(x,y)  \phi_{i}(y)f(y)+\nabla u_i(x,k)(\Delta_{\R^{n_{i}}\times
		\mathcal{M}_{i}} + k^{2})^{-1}(x_i^\circ,y)\phi_i(y)f(y)\r|dy.
\end{align}
%Now, if we compute the kernel in the first part of the above operator at $x=x_i^\circ$ for some $i$, then the unboundedness of the following operator will imply that of \eqref{G_1+G_3}.
%Now, if we compute the kernel in the first part of the above operator at $x=x_i^\circ$ for some $i$, 
Now, if we compute the first part of the above kernel at $x=x_i^\circ$ for some $i$, then to verify estimates for \eqref{G_1+G_3}, %it is enough to show that
it is enough to get lower bounds like \eqref{a20} for the operator corresponding to the following kernel
\begin{align}\label{G_1+G_3_x_i}
    k\l|\nabla \phi_{i}(x)+\nabla u_i(x,k)\r| \int_\M \l|(\Delta_{\R^{n_{i}}\times
		\mathcal{M}_{i}} + k^{2})^{-1}(x_i^\circ,y)\phi_i(y)f(y)\r|dy.
\end{align}
Thus, we are left to show that the $L^p$ norm of the operator with the kernel described in \eqref{G_1+G_3_x_i} is greater than $ck^{-(1-n^*/p)} $ for some $c>0$ and all $1>k>0$.

Let $a_k(x)=\nabla\phi_{i}(x)+\nabla u_i(x,k) $, then adding and subtracting  $\nabla u_i(x,0)$ in $a_k(x)$ we obtain
$$a_k(x)=\nabla (u_i(x,k)-u_i(x,0))+(\nabla \phi_i(x)+\nabla u_i(x,0)).$$
It has been shown in \cite{HS2} that the first term of $a_k(x)$ is uniformly bounded on $L^p(\M)$ for all $p\in(1,\infty)$. Therefore, the required estimates for  \eqref{G_1+G_3_x_i} will be implied by the lower bounds for 
\begin{align}
    k\l|\nabla \phi_{i}(x)+\nabla u_i(x,0)\r| \int_\M \l|(\Delta_{\R^{n_{i}}\times
		\mathcal{M}_{i}} + k^{2})^{-1}(x_i^\circ,y)\phi_i(y)f(y)\r|dy.
\end{align}
Let $a(x)=\l|\nabla \phi_{i}(x)+\nabla u_i(x,0)\r|$ and $f$ be a positive function. Also, let
$$Rf(x)=k a(x) \int_\M (\Delta_{\R^{n_{i}}\times
		\mathcal{M}_{i}} + k^{2})^{-1}(x_i^\circ,y)\phi_i(y)f(y)dy.$$
Then, by \cite[Corollary 2.3]{BaSi}, there exists a constant $C>0$ such that 
\begin{align*}
    Rf(x)&\geq C ka(x)\int_\M d(x_i^\circ,y)^{2-n_i}e^{-ckd(x_i^\circ,y)}\phi_i(y)f(y)dy.
\end{align*}
%Let $g(y)=d(x_i^\circ,y)^{2-n_i}e^{-ckd(x_i^\circ,y)}$, then applying H{\"o}lder's equality on the above integral in $L^1$ we see that $f\in L^p$ if $g\in L^q$, where $\frac{1}{p}+\frac{1}{q}=1$, and $|f|^p\simeq|g|^q$ almost everywhere. 

Thus,
\begin{align*}
Rf(x)&\gtrsim ka(x)\int_\M g(y)f(y)\phi_i(y)dy\simeq ka(x)\int_{\R^{n_i}\times\M_i\setminus K_i} g(y)f(y)dy,
\end{align*}
where $g(y)=d(x_i^\circ,y)^{2-n_i}e^{-ckd(x_i^\circ,y)}$.
Now,
\begin{align*}
||R||_{p\to p} &\gtrsim k\l(\int_\M |a(x)|^pdx\r)^{1/p}||g||_q
 \gtrsim k||g||_q\text{ (since } a(x)\not\equiv 0 )\\
 &\simeq k \gamma_\beta(k), \text{ where }\beta=-(n_i-2)q+n_i\\
 &\simeq k^{\frac{n_i}{p}-1}.
\end{align*}
That is,  $||R||_{p\to p}\gtrsim k^{\frac{n^*}{p}-1}$ for $0<k\leq1$ and $p\in[n^*,\infty]$, which further implies that \eqref{szescNeg} holds for all $t\geq1$ and $p\in [n^*,\infty]$.
This ends the proof of Theorem~\ref{GpThm}. \\

%\com{HS: Need a Lemma for the $\mathbb{R}^n$ case for maximal functions.}

\subsection{Proof of Theorem~\ref{maxThm}}

%\begin{proof}
The boundedness of the horizontal maximal operator follows from \cite[Theorem~3]{DLS}, \eqref{Res-Exp}, and the fact that we can write
\begin{equation}
    t\Delta(1+t\Delta)^{-m}=(1+t\Delta)^{-(m-1)}-(1+t\Delta)^{-m}.
\end{equation}
Therefore, we are interested in obtaining bounds for $M^{res,\nabla}_m$.
		Clearly, it follows from \eqref{szesc} that 
$$
\|M_m^{res,\nabla}f(x)\|_{p \to p } =\infty
$$
for all $p > n^*$. 

Again replacing $t$ by $\frac{1}{k^2}$ we obtain
\begin{align*}
    M_{m}^{res,\nabla} f(x)=\sup _{t>0}\left|\sqrt{t} \nabla(1+t \Delta)^{-m} f(x)\right|=\sup _{k>0}\left|k^{2m-1} \nabla\left(k^{2}+\Delta\right)^{-m} f(x)\right|.
\end{align*}
By the same argument as in Theorem~\ref{GpThm}, we focus on the case when $t>1$, that is, $0<k<1$.
\begin{align*}
    \left\|M^{res,\nabla}_m f\right\|_{p \to p} &=\left\|\sup_{0<k<1}\left|k^{2m-1} \nabla\left(k^{2}+\Delta\right)^{-m} f\right|\;\right\|_{p \to p} \\
&=\left\|\sup_{0<k<1}\left|k^{2m-1} \nabla \sum_{j=1}^{4} H_{j}(k) f\right|\;\right\|_{p \to p} \leq \sum_{j=1}^{4}\left\|\sup _{0<k<1}\left|k^{2m-1} \nabla H_{j}(k) f\right|\;\right\|_{p \to p}.
\end{align*}
We start with the boundedness of the $H_1(k)$ term.
\begin{align*}
\left\|\sup _{0<k<1}\left|k^{2m-1} \nabla H_1(k) f\right|\right\|_{p \rightarrow p}&=
   \l|\l| \sup_{0<k<1}\Big| k^{2m-1} \int_{\M} \nabla\Big(\left(\Delta_{\mathbb{R}^{n_i}\times \M_{i}}+k^{2}\right)^{-m}(\cdot, y) \phi_{i}(\cdot)\phi_i(y)\Big) f(y) d y\Big|\;\r|\r|_{p\to p}\\
&=\l|\l|\phi_{i}(\cdot) \sup_{0<k<1}\Big| k^{2m-1} \int_{\R^{n_i}\times\M_i} \nabla\left(\Delta_{\mathbb{R}^{n_i}\times \M_{i}}+k^{2}\right)^{-m}(\cdot, y) \phi_i(y) f(y) d y\Big|\;\r|\r|_{p\to p}\\
&\;\;\;+\l|\l|\nabla \phi_i(\cdot) \sup _{0<k<1}\Big| k^{2m-1} \int_{\R^{n_i}\times\M_i}\left(\Delta_{\mathbb{R}^{n_i}\times \M_{i}}+k^{2}\right)^{-m}(\cdot, y) \phi_i(y) f(y) d y\Big|\;\r|\r|_{p\to p}\\
&=I_{1}+I_{2}.
\end{align*}
Boundedness of $I_1$ holds for all $1<p\leq\infty$ and is weak-type (1,1) by Lemma~\ref{MaxLemma}.

For $I_2$, following on similar lines as in $H_{1,b}$ of Theorem~\ref{GpThm}, we take $\chi_D$ as the characteristic function of the set $D=\{(x,y)\in\mathcal{M}^2: d(x,y)\leq1\}$. Then, by \cite[Proposition~2.2]{BaSi}, we have
\begin{equation*}
    \l|\l|\nabla\phi_i(\cdot)\chi_D(\cdot,y)\int_{\R^{n_i}\times\M_i}\sup_{k}|k^{2m-1}(\Delta_{\R^{n_i}\times\M_i}+k^2)^{-m}(\cdot,y)|\phi_i(y) f(y)dy\r|\r|_{p\to p}\lesssim ||f||_p .
\end{equation*}
Consider now
\begin{align*}
    &\l|\l|\nabla \phi_i(\cdot)(1-\chi_D) \sup _{0<k<1}\Big| k^{2m-1} \int_{\R^{n_i}\times\M_i}\left(\Delta_{\mathbb{R}^{n_i}\times \M_{i}}+k^{2}\right)^{-m}(\cdot, y) \phi_i(y) f(y) d y\Big|\;\r|\r|_{p\to p}\\
    &=\l(\int_{\M}\left|\nabla \phi_i(x) (1-\chi_D)\sup _{0<k<1}\Big| k^{2m-1} \int_{\R^{n_i}\times\M_i}\left(\Delta_{\mathbb{R}^{n_i}\times \M_{i}}+k^{2}\right)^{-m}(x, y) \phi_i(y) f(y) d y\Big|\right|^{p} d x\r)^{1/p}\\  
     &\lesssim ||f||_p\l(\int_{\M}\Bigg|\nabla \phi_i(x) \sup _{0<k<1}\Big| k \l(\int_{\R^{n_i}\times\M_i\setminus K_i}\l|d(x,y)^{-(n_i-2)}e^{-ckd(x,y)}\r|^{p'}dy\r)^{1/p'}\Big|\Bigg|^p dx\r)^{1/p}\\
     &\lesssim ||f||_p\l(\int_{\text{ supp }\nabla \phi_i}\Bigg|\sup _{0<k<1}\Big| k \l(\int_{\R^{n_i}\times\M_i\setminus K_i}\l|d(x,y)^{-(n_i-2)}e^{-ckd(x,y)}\r|^{p'}dy\r)^{1/p'}\Big|\Bigg|^p dx\r)^{1/p}\\
     &\lesssim ||f||_p \sup _{0<k<1} k \l(\int_{1}^\infty \l|r^{-(n_i-2)}e^{-ckr}\r|^{p'}r^{n_i-1}dr\r)^{1/p'}\;\text{(since }\nabla\phi_i\text{ is compactly supported)}\\
     &\lesssim ||f||_p\sup _{0<k<1} k \l(\int_{k}^\infty \frac{r^{-(n_i-2)p'+n_i-1}}{k^{-(n_i-2)p'+n_i}}e^{-r}dr\r)^{1/p'},
\end{align*}
where the third inequality follows again by \cite[Proposition~2.2]{BaSi} and the fact that $\phi_i$ is supported outside a compact set.
The above supremum is finite by the $H_3(k)$ term of Theorem~\ref{GpThm} and thus bounded for $1<p\leq n_i$. Hence, $I_2$ is bounded for $1<p\leq n_i$. \\
It is also bounded when $p=1$. Indeed,
\begin{align*}
    I_2&=\int_{\M}\left|\nabla \phi_i(x) \sup _{0<k<1}\Big| k \int_{\R^{n_i}\times\M_i}\left(\Delta_{\mathbb{R}^{n_i}\times \M_{i}}+k^{2}\right)^{-1}(x, y) \phi_i(y) f(y) d y\Big|\right|d x\\
     &\lesssim ||f||_1\int_{\text{ supp }\nabla \phi_i}\Bigg|\sup _{0<k<1}\Big| k \sup_{y\in\M}\l|d(x,y)^{-(n_i-2)}e^{-ckd(x,y)}\phi_i(y)\r|\Big|\Bigg| dx\\
     &\simeq ||f||_1\int_{\text{ supp }\nabla \phi_i}\Big|  \sup_{y\in\R^{n_i}\times\M_i\setminus K_i}\sup _{0<k<1}\l|k d(x,y)^{-(n_i-2)}e^{-ckd(x,y)}\r|\Big|dx\simeq ||f||_1\int_{\text{ supp }\nabla \phi_i} \sup_{y\in\R^{n_i}\times\M_i\setminus K_i}\l| d(x,y)^{-(n_i-1)}\r|dx\\
     &\lesssim ||f||_1\sup_{x\in\text{ supp }\nabla \phi_i} \sup_{y\in\R^{n_i}\times\M_i\setminus K_i}\l| d(x,y)^{-(n_i-1)}\r|\lesssim||f||_1.
\end{align*}

Since the $H_2(k)$ term is again a pseudodifferential operator, its boundedness holds in the same way as in Theorem~\ref{GpThm} for all $1\le p \le \infty$. \\

For $H_3(k)$ we follow the similar procedure as done in Theorem~\ref{GpThm}.

{\it Case $p>\frac{n_i}{n_i-1}$.}
We have 
\begin{align}\label{g3imax}
&\l|\l| \sup_{0<k< 1}|k^{2m-1}\nabla H_3(k)f|\;\r|\r|_{p}=\l|\l| \sup_{0<k< 1}|k^{2m-1}\nabla \int_\M H_3(k)(\cdot,y)f(y)dy|\;\r|\r|_{p}\\
&=\l(\int_{\M}\sup_{0<k<1}\left|k^{2m-1}\int_{\M} \nabla_x H_3(k)(x,y)f(y) dy\right|^p dx\r)^{1/p}\notag\\
    &\lesssim \l(\int_{\M}\sup_{0<k<1}\left|k\int_{\M} \omega_1(x,k)\omega_2(y,k) f(y) dy\right|^p dx\r)^{1/p}\notag\\
    &\lesssim ||f||_p \l(\int_{\M}\sup_{0<k<1} |\omega_1(x,k)|^p k^p \left(\int_{\M} |\omega_2(y,k)|^{p'} dy\right)^{p/p'} dx\r)^{1/p}\notag\\
    &\lesssim ||f||_p\l(\int_{\R^{n_i}\times\M_i\setminus K_i} \sup_{0<k<1} \l|\left<d(x_i^\circ,x)\right>^{-(n_i-1)}e^{-ck(d(x_i^\circ,x)}\r|^p k^p \left( \int_{\M} |\omega_2(y,k)|^{p'} dy\right)^{p/p'} dx\r)^{1/p}\notag\\
    &\lesssim ||f||_p\l(\int_{\R^{n_i}\times\M_i\setminus K_i} \l|\left<d(x_i^\circ,x)\right>^{-(n_i-1)}\r|^p dx\r)^{1/p} \sup_{0<k<1}\left(k\gamma_\beta(k)\right)\notag,
\end{align}
where the first integral is obtained by controlling $e^{-ckd(x_i^\circ,x)}$ by $1$ and is finite for all $(n_i-1)p>n_i$, i.e., $p>\frac{n_i}{n_i-1}$. The supremum, on the other hand, is finite by the $H_3(k)$ term of Theorem~\ref{GpThm}. This gives the boundedness for $\frac{n^*}{n^*-1} <p\leq n_i$.
Hence, the $H_3(k)$ term is bounded here for all $\frac{n^*}{n^*-1} <p\leq n^*$.

{\it Case $p=1$.}
In the next step we shall prove that $H_3(k)$ is of weak-type $(1,1)$. % By interpolation it will show the continuity for the whole range $1<p \le n^* $.
When $p=1$ we note that 
$$
\sup_k| k\omega_1(x,k)| \le C \omega(x),
$$
where $\omega(x)=\l< d(x_i^\circ,x)\r>^{-n_i}$ for 
$x \in \R^{n_i}\times\M_i$ and $\omega(x)=1$ for $x\in K$. 
In addition, we note that the $L^\infty$ norm of $\omega_2(y,k)$ is uniformly (independently of $k$) bounded. 
Since the function $d(x_i^\circ,x)^{-n_i}$ is in $L^{1,\infty}$, we obtain that the maximal $H_3(k)$ term is weak-type (1,1).
Next, by  interpolation it follows that the $H_3$ part  is continuous  
for the whole range $1<p \le n^* $.

\medskip
The analog of the argument for $H_3$ shows the boundedness of the $H_4(k)$ term but for $1<p\leq \infty$. In fact, we obtain the weak-type (1,1) estimates  by a straightforward modification of the argument from \eqref{g3imax}. Thus, the boundedness of the maximal function is obtained on $L^p$ spaces for $1<p\leq n^*$ by combining the above $H_j(k)$ terms $(1\leq j\leq 4)$, and is weak-type (1,1).\\

%\end{proof}

We now prove the Fefferman-Stein maximal inequality that is in our case  Theorem~\ref{FSThm}.
\subsection{Proof of Theorem~\ref{FSThm}}
%\com{AS: I expect that the proof has the same structure as before}
%\begin{proof}
First we note that if we consider supremum taken in \eqref{VMF} only over $0<t <1$
then  the maximal operator can essentially be discussed in the same way as Proposition~\ref{FS}. Therefore, we only
consider the range $t\ge 1$ or equivalently $0<k \le 1$.

Similarly, as before, we consider $H_{1,a}$--- when the derivative hits the resolvent. That is, we need to compute the $L^p$ norm of 
 \begin{equation}
     \phi_i k^{2m-1} \nabla(\Delta_{\R^{n_i} \times \mathcal{M}_i} + k^2)^{-m}   \phi_i.
 \end{equation}
The estimates required here for part $H_{1,a}$ follows  directly from Proposition \ref{FS}.
We discuss the part $H_{1,b}$ later together with $H_3$.
Recall that  the  $H_2$ term can be estimated by 
\begin{equation*}
    |k^{2m-1}\nabla_x H_2(k)(x,y)|\lesssim k^{\frac{1}{2}} d(x,y)^{\frac{1}{2}-N}.
\end{equation*}
It follows that we can replace supremum over $0<k \le 1$ for $H_2$ and so we end up with 
the linear operator bounded by $d(x,y)^{\frac{1}{2}-N}$. In consequence, we can use the same
argument as in the proof of Theorem~\ref{GpThm} and Lemma~\ref{FSLem} to verify continuity for all $1<p<\infty$ for the $H_2$ term .

First we observe that $$\sup_{k\le 1}k(1+r)^{1-n}e^{-rk} \le (1+r)^{-n}.$$ Then, 
assuming  that $f(x) \ge 0$ and that  $f\in L^1(\M)$, we have 
\begin{align}\label{H3sup}
    \sup_k|k^{2m-1}\nabla H_3(k)f(x)|&=\sup_{k}\l|k^{2m-1}\int_\M\nabla H_3(k)(x,y)f(y)dy\r|\\
    &\lesssim \sup_k\l|k\o_1(x,k)\int_\M\o_2(y,k)f(y)dy\r|\nonumber\\
    &\lesssim \sup_k\l|k\o_1(x,k)\r|\sup_{y\in\R^{n_i}\times\M_i\setminus K_i} \l|d(x_i^\circ,y)^{-(n_i-2)}\r|\int_\M f(y)dy\nonumber\\
    &\lesssim \o(x) \int_\M f(y)dy\nonumber,
\end{align}
where $\omega(x)=\l<d(x_i^\circ,x)\r>^{-n_i}$ for $x\in\R^{n_i}\times\M_i\setminus K_i$ and $\omega(x)=1$ for $x\in K$.\\
Now choose a sequence $\bar{f}=\{f_j\}_{j=1}^\infty$. Note that, without loss of generality, we can assume $f_j(x) \ge 0$. Then 
\begin{align*}
\sum_j\l(\sup_k|k^{2m-1}\nabla H_3(k){f_j}(x)|\r)^2
\le  C \omega(x)^2   \sum_j\l|\int_{\M} 
{f_j}dy\r|^2=C \omega(x)^2 \l|\Lambda( \bar{f})\r|^2,
\end{align*}
where $\Lambda(f)=\int_{\M} 
f(y)dy$. Now $\Lambda$ is a bounded linear function on $L^1(\M)$ so we can apply Lemma~\ref{FSLem}. This  validates weak-type $(1,1)$.

Observe next that 
$$
\omega_1(x,k) \le C \tilde{\omega}(x)  \quad
\mbox{and that} \quad \sup_k k\omega_2(x,k) \le C \tilde{\omega}(x), $$
where $ \tilde{\omega}(x) =\l< d(x_i^\circ,x)\r>^{1-n_i}$ for 
$x \in \R^{n_i}\times\M_i\setminus K_i$ and $\tilde{\omega}(x)=1$ for $x\in K$. 
To consider the range $2<p<n^*$, we slightly rearrange \eqref{H3sup} and we note that if $f(x) \ge 0$, then
%To prove the relation in \eqref{H3sup} for 
%$2<p<n^*$ we note that if $f(x) \ge 0$, then
\begin{align*}
    \sup_k|k^{2m-1}\nabla H_3(k)f(x)|&\lesssim \sup_k\l|\o_1(x,k)\int_\M k\o_2(y,k)f(y)dy\r|\nonumber\\
    &\lesssim \tilde{\o}(x) \int_\M \tilde{\o}(y)f(y)dy\nonumber.
\end{align*}
Thus, for $\bar{f}=\{f_j\}_{j=1}^\infty$ with each $f_j(x)\geq0$, we have
\begin{align*}
   \sum_{j}\l( \sup_k|k^{2m-1}\nabla H_3(k)f_j(x)|\r)^2\lesssim\sum_j\l|\tilde{\o}(x) \int_\M \tilde{\o}(y)f_j(y)dy\r|^2\simeq|\tilde{\Lambda}(\bar{f})|^2,
\end{align*}
where $\tilde{\Lambda}(f)= \tilde{\omega}(x) \int_{\M} \tilde{\omega}(y)f(y)dy$. Note that 
the operator $\tilde{\Lambda}$ acts continuously  on $L^p(\M)$ if and only
if $(n^*)' < p < n^*$. 
Using interpolation with weak-type 
$(1,1)$ we show boundedness of the $H_3$ term
for all $1 < p < n^*$. For the $H_{1,b}$ term, observe that
\begin{align*}
    \nabla\phi_i(x)&\sup_{0<k<1}\l|k^{2m-1}\int_{\R^{n_i}\times\M_i}(\Delta_{\R^{n_i}\times\M_i}+k^2)^{-m}(x,y)\phi_i(y)f(y)dy\r|\\
&\lesssim\nabla\phi_i(x)\sup_{0<k<1}\l|k\int_{\R^{n_i}\times\M_i\setminus K_i}d(x,y)^{-(n_i-2)}e^{-ckd(x,y)}f(y)dy\r|\\&\lesssim \int_{\R^{n_i}\times\M_i\setminus K_i} \l<d(x,y)\r>^{-(n_i-1)}f(y)dy,\; \text{ for }x\in \text{supp}\nabla\phi_i \text{ (since }\nabla\phi_i(x)\text{ is compactly supported)}.
\end{align*}
Just like in the $H_3$ term, the above integral is also controlled by a bounded linear functional on $L^p$ if and only if $1<p<n^*$, and is weak type~(1,1) bounded on $L^1$. Again, for a sequence $\bar{f}=\{f_j\}_{j=1}^\infty$, the $H_{1,b}$ term satisfies the Fefferman-Stein inequality.
The proof for the $H_4$ term follows on the same lines as $H_3$ and as done in Theorems~\ref{GpThm} and \ref{maxThm}. However, the boundedness for $H_4$ holds for $1<p<\infty$ with weak type~(1,1) at $L^1$. Clubbing the boundedness of each $H_i$ term, we obtain the Fefferman-Stein maximal inequality for the vertical maximal function on $L^p(\M)$ for $1<p<n^*$ with weak type(1,1) estimates at $L^1$. The unboundedness for $p\geq n^*$ is shown in Corollary~\ref{cor2}.

%\end{proof}

\section{Square functions and R-boundedness} \label{R}

In this section we show that the boundedness of vertical maximal function implies the boundedness of the vertical square function. This is possible with the help of the R-boundedness %(Rademacher boundedness) 
of the set $\{\sqrt{t}\nabla(1+t\Delta)^{-m}:t>0\}$. Before we jump to our results, we recall some basic definitions and known results required for our proof.

\begin{definition}(\cite[Definition 8.1]{ABS2})\label{R-Bdd_Equiv}
Let $X$ and $Y$ be Banach spaces and let $\mathcal{T}\subseteq B(X,Y)$ be a family of bounded operators. 
\begin{enumerate}
    \item[(a)] R-boundedness: $\mathcal{T}$ is said to be R-bounded if there exists a constant $C\geq 0$ such that for all finite sequences $(T_j)_{j=1}^J$ in $\mathcal{T}$ and $(x_j)_{j=1}^J$ in $X$,
    \begin{equation*}
        \l|\l|\sum_{j=1}^J\epsilon_j T_j x_j\r|\r|_{L^2(\Omega,Y)}\leq C\l|\l|\sum_{j=1}^J\epsilon_j x_j\r|\r|_{L^2(\Omega,X)}
    \end{equation*}
    or equivalently,
    \begin{equation*}
\mathbb{E}\l|\l|\sum_{j=1}^J\epsilon_j T_j x_j\r|\r|_{L^2(\Omega,Y)}\leq C\mathbb{E}\l|\l|\sum_{j=1}^J\epsilon_j x_j\r|\r|_{L^2(\Omega,X)},
   \end{equation*}
   where $\mathbb{E}$ is the usual expectation and $(\epsilon_j)_{j\geq1}$ is a Rademacher sequence on a fixed probability space $(\Omega,\mathbb{P})$. 
   \item[(b)] $\ell^2$-boundedness: If $X$ and $Y$ are Banach lattices, then $\mathcal{T}$ is said to be $\ell^2$-bounded if there exists a constant $C>0$ such that for all finite sequences $(T_j)_{j=1}^J$ in $\mathcal{T}$ and $(x_j)_{j=1}^J$ in $X$,
   \begin{equation*}
    \l|\l|\Bigg(\sum_{j=1}^J|T_j x_j|^2\Bigg)^{1/2}\r|\r|_{L^2(\Omega,Y)}\leq C\l|\l|\Bigg(\sum_{j=1}^J| x_j|^2\Bigg)^{1/2}\r|\r|_{L^2(\Omega,X)}.
    \end{equation*}
\end{enumerate}
The least admissible constants here are called the R-bound and the $\ell^2$-bound, and are denoted by $\mathcal{R}(\mathcal{T})$ and $\ell^2(\mathcal{T})$, respectively.

\end{definition}

\begin{remark}
In $(a)$ and $(b)$ of Definition \ref{R-Bdd_Equiv} we can replace $L^2(\Omega,Y)$ by $L^p(\Omega,Y)$, and $L^2(\Omega,X)$ by $L^q(\Omega,X)$ for $1\leq p,q<\infty$. In such situations we obtain an equivalent definition of the R-boundedness, but with a possibly different constant that we denote by $\mathcal{R}_{p,q}(\mathcal{T})$. For $p=q$, it is convenient to write $\mathcal{R}_{p}(\mathcal{T}):=\mathcal{R}_{p,p}(\mathcal{T})$. We also refer to \cite[Proposition~8.1.5]{ABS2} for the definition of R-boundedness on $L^p$ spaces for $1\leq p<\infty$.
\end{remark}

\begin{proposition}\label{Rmk}
When $X$ and $Y$ are $L^p$ spaces with $p\in[1,\infty)$, then R-boundedness and $\ell^2$-boundedness are equivalent, see \cite[Theorem~8.1.3]{ABS2}.
\end{proposition}

\begin{theorem}\label{RBddThm}
The Fefferman-Stein maximal inequality in \eqref{FSeqn} implies that the
 set $\{\sqrt{t}\nabla (1+t\Delta)^{-m}:t>0\}$ is R-bounded on $L^p(\M)$ for $1<p<n^*$.
\end{theorem}
\begin{proof}
Consider the following set 
\begin{equation*}
    \mathcal{T}:=\{T(t)=\sqrt{t}\nabla (1+t\Delta)^{-m}\in \mathcal{L}(L^p(\Omega))\text{ for }t>0: |T(t)f|\leq M_m^{res,\nabla}f \text{ for all }f\in L^p(\Omega)\}.
\end{equation*}
By Proposition \ref{Rmk}, it is enough to show $\ell^2$ boundedness of the above set to prove its R-boundedness.
We have
\begin{align*}
\bigg|\bigg|\bigg(\sum_{j=1}^J|T(t_j)f_j|^2\bigg)^{1/2}\bigg|\bigg|_{L^p(\Omega)}&\leq\bigg|\bigg|\bigg(\sum_{j=1}^J(M_m^{res,\nabla}f_j)^2\bigg)^{1/2}\bigg|\bigg|_{L^p(\Omega)}\\
    &\leq A_{p,\M}\bigg|\bigg|\bigg(\sum_{j=1}^J|f_j|^2\bigg)^{1/2}\bigg|\bigg|_{L^p(\Omega)}.
\end{align*}
The last inequality follows from Theorem~\ref{FSThm}.
\end{proof}

\begin{coro}\label{cor1}
The Fefferman-Stein maximal inequality implies the boundedness of the resolvent based vertical square functions $S^{res,\nabla}_m(f):=\l(\int_0^\infty|\sqrt{t}\nabla(1+t\Delta)^{-m}f|^2\frac{dt}{t}\r)^{1/2}$ for any order $m$ and the semigroup based vertical square function $S^{exp,\nabla}(f):=\l(\int_0^\infty|\sqrt{t}\nabla e^{-t\Delta}f|^2\frac{dt}{t}\r)^{1/2}$ on $L^p(\M)$ for $1~<~p~<~n*$.
\end{coro}

\begin{proof}
Theorem~\ref{RBddThm} shows that the Fefferman-Stein maximal inequality of the vertical maximal functions implies $R$-boundedness of the set $\{\sqrt{t}\nabla (1+t\Delta)^{-m}:t>0\}$ on $L^p$ for $1<p< n^*$. By \cite[Proposition~2.2]{CO20}, R-boundedness of the set $\{\sqrt{t}\nabla (1+t\Delta)^{-m}:t>0\}$ is equivalent to R-boundedness of the set $\{\sqrt{t}\nabla e^{-t\Delta}:t>0\}$ on $L^p$. Interestingly, from the R-boundedness of the latter set,  \cite[Theorem~4.1]{CO20} gives the boundedness of more general square functions which satisfy some decaying conditions. Both $S^{exp,\nabla}$ and $S^{res,\nabla}_m$ satisfy those conditions on $L^p$ spaces, and hence they are bounded for $1<p<n^*$.
\end{proof}

\begin{coro}\label{cor2}
The Fefferman-Stein maximal inequality \eqref{FSeqn} in Theorem~\ref{FSThm} fails for $p\geq n^*$. The vertical square function operator $S^{exp,\nabla}$ is also unbounded in the same range.
\end{coro}

\begin{proof}
 We know from \cite[Theorem~1.1]{BaSi} that the resolvent based square function is unbounded on $L^p(\M)$ for $p \ge n^*$. This and Corollary \ref{cor1} imply that the Fefferman-Stein maximal inequality fails in the same range. Also, since the $R$-boundedness of the set $\{\sqrt{t}\nabla (1+t\Delta)^{-m}:t>0\}$ is false for $p \ge n^*$ by the same argument, we obtain negative result for $S^{exp,\nabla}$ from \cite[Proposition~2.2 and Theorem~4.1]{CO20}.
\end{proof}

\begin{proposition}\label{prop4.2}
    On $L^p(\M)$, the vertical square function operator $S^{exp,\nabla}$ is bounded if and only if $1~<~p~<~n^*$.
\end{proposition}
\begin{proof}
  The proof is a direct consequence of Corollaries \ref{cor1}~and~\ref{cor2}. 
\end{proof}
Note that at this point we do not have any good estimates for the kernel of $\nabla e^{-t\Delta}$ on $L^p(\M)$, and hence we do not have any direct proof to show the boundedness of $S^{exp,\nabla}$ on $L^p(\M)$.\\

The following theorem is the resolvent based version of the square function in \cite[Theorem~3.1]{CO20}. The proof of it follows exactly in the same way as that of \cite[Theorem~3.1]{CO20}, however, the R-boundedness of the resolvent sets require the boundedness of two resolvent based square functions of consecutive exponents. We have included the proof here to give an equivalence argument between the semigroup based square functions and the resolvent based square functions.
Note that this theorem works on a general metric measure space $\Omega$. However, on manifolds with ends, the boundedness of square function operators $S^{exp,\nabla}$ and $S^{res,\nabla}_m$ coincide under the assumption that $2m+2\leq \min{n_i}$. We explain this in Remark \ref{RSF-SSF} below.

In the proof of the below theorem we use the notation of the norm with subscript $p$ to denote the norm with subscript $L^p(\M)$ for simplicity.

\begin{theorem}\label{ResSF}
Let $p\in(1,\infty)$ and $m\geq1$. If $S^{res,\nabla}_m$ and $S^{res,\nabla}_{m+1}$ are bounded on $L^p(\M)$, then the set $\{t\nabla (1+t^2\Delta)^{-m}, t>0\}$ is R-bounded on $L^p(\M).$ 
\end{theorem}

\begin{proof}
The proof follows on the same lines of \cite[Theorem 3.1]{CO20}. We shall show if $S^{res,\nabla}_m$ and $S^{res,\nabla}_{m+1}$ are bounded on $L^p(\M)$ for some $m$ then $\{t\nabla (1+t^2\Delta)^{-m}, t>0\} $ is R-bounded on $L^p(\M)$. %A general result which, in particular, proves the converse of this theorem is given in \cite[Theorem 4.1]{CO20}.
Let $t_j\in(0,\infty)$ and $f_j\in L^p(\M)$ for $j=1,...,J$. We begin by estimating the quantity $I:=\E|\sum_j\e_jt_j\nabla(1+t_j^2\Delta)^{-m}f_j|^2$. Using the fact that the Rademacher variables are independent we obtain
\begin{align*}
    I&=-\int_0^\infty\frac{d}{dt}\E|\nabla(1+t^2\Delta)^{-m}\sum_j\e_jt_j(1+t_j^2\Delta)^{-m}f_j|^2dt\\
    &=2m\int_0^\infty\E\Bigg[\Big(\nabla(1+t^2\Delta)^{-2m-1}(2t\Delta)\sum_j\e_jt_j(1+t_j^2\Delta)^{-m}f_j\Big)\cdot\Big(\nabla\sum_j\e_jt_j(1+t_j^2\Delta)^{-m}f_j\Big)\Bigg]dt\\
    &=2m\int_0^\infty\E\Bigg[\Big(\nabla(1+t^2\Delta)^{-m}\sum_j\e_jt_j(1+t_j^2\Delta)^{-m}f_j\Big)\cdot\Big(\nabla(1+t^2\Delta)^{-m-1}(2t\Delta)\e_jt_j(1+t_j^2\Delta)^{-m}f_j\Big)\Bigg]dt\\
    &=4m\int_0^\infty\E\Bigg[\Big(\nabla t(1+t^2\Delta)^{-m}\sum_j\e_j(1+t_j^2\Delta)^{-m}f_j\Big)\cdot\Big(\nabla t(1+t^2\Delta)^{-(m+1)}\e_jt_j^2\Delta(1+t_j^2\Delta)^{-m}f_j\Big)\Bigg]\frac{dt}{t}\\
    &=4m\int_0^\infty\E\Bigg[\Big(\nabla t(1+t^2\Delta)^{-m}\sum_j\e_j(1+t_j^2\Delta)^{-m}f_j\Big)\cdot\Big(\nabla t(1+t^2\Delta)^{-(m+1)}\sum_j\e_jt_j^2\Delta(1+t_j^2\Delta)^{-m}f_j\Big)\Bigg]dt.
\end{align*}
By the Cauchy-Schwartz inequality,
\begin{align*}
    I&\leq 4m\int_0^\infty\Bigg(\E|t\nabla(1+t^2\Delta)^{-m}\sum_j\e_j(1+t_j^2\Delta)^{-m}f_j|^2\Bigg)^{1/2}\Bigg(\E|t\nabla(1+t^2\Delta)^{-(m+1)}\sum_j\e_jt_j^2\Delta(1+t_j^2\Delta)^{-m}f_j|^2\Bigg)^{1/2}\frac{dt}{t}\\
    &\leq 2m\Bigg(\int_0^\infty\E|t\nabla(1+t^2\Delta)^{-m}\sum_j\e_j(1+t_j^2\Delta)^{-m}f_j|^2\frac{dt}{t}+\int_0^\infty\E|t\nabla(1+t^2\Delta)^{-(m+1)}\sum_j\e_jt_j^2\Delta(1+t_j^2\Delta)^{-m}f_j|^2\frac{dt}{t}\Bigg).
\end{align*}
Since both $S^{res,\nabla}_m$ and $S^{res,\nabla}_{m+1}$ are bounded, we get
\begin{equation*}
    I\lesssim_{m}\E\Bigg[\Big(S^{res,\nabla}_m\Big(\sum_j\e_j(1+t_j^2\Delta)^{-m}f_j\Big)\Big)^2\Bigg]+\E\Bigg[\Big(S^{res,\nabla}_{m+1}\Big(\sum_j\e_jt_j^2\Delta(1+t_j^2\Delta))^{-m}f_j\Big)\Big)^2\Bigg].
\end{equation*}
Looking at $S^{res,\nabla}_m$ and $S^{res,\nabla}_{m+1}$ as the norm in $L^2((0,\infty),\frac{dt}{t})$ we obtain
\begin{equation*}
    \E\Bigg[\Big(S^{res,\nabla}_m\Big(\sum_j\e_j(1+t_j^2\Delta)^{-m}f_j\Big)\Big)^2\Bigg]=\E\Bigg|\Bigg|\sum_j\e_jt\nabla(1+t^2\Delta)^{-m}(1+t_j^2\Delta)^{-m}f_j\Bigg|\Bigg|^2_{L^2((0,\infty),\frac{dt}{t})}
\end{equation*}
and
\begin{equation*}
    \E\Bigg[\Big(S^{res,\nabla}_{m+1}\Big(\sum_j\e_j(1+t_j^2\Delta)^{-m}f_j\Big)\Big)^2\Bigg]=\E\Bigg|\Bigg|\sum_j\e_jt\nabla(1+t^2\Delta)^{-(m+1)}\sum_j\e_jt_j^2\Delta(1+t_j^2\Delta)^{-m}f_j\Bigg|\Bigg|^2_{L^2((0,\infty),\frac{dt}{t})}.
\end{equation*}
Now, by the Kahane inequality
\begin{equation*}
    c_{p,m}\sqrt{I}\leq\Bigg|\E\Bigg[\Big(S^{res,\nabla}_m\Big(\sum_j\e_j(1+t_j^2\Delta)^{-m}f_j\Big)\Big)^p\Bigg]\Bigg|^{1/p}+\Bigg|\E\Bigg[\Big(S^{res,\nabla}_{m+1}\Big(\sum_j\e_jt_j^2\Delta(1+t_j^2\Delta)^{-m}f_j\Big)\Big)^p\Bigg]\Bigg|^{1/p}
\end{equation*}
for some constant $c_{p,m}>0$. Now, using the assumption that $S^{res,\nabla}_m$ and $S^{res,\nabla}_{m+1}$ are bounded on $L^p(\M)$, we obtain
\begin{align*}
\Big|\Big|\sqrt{I}\Big|\Big|_p&\lesssim_{p,m}\bigg|\E\bigg|\bigg|\sum_j\e_j(1+t_j^2\Delta)^{-m}f_j\bigg|\bigg|_p^p\bigg|^{1/p}+\bigg|\E\bigg|\bigg|\sum_j\e_jt_j^2\Delta(1+t_j^2\Delta)^{-m}f_j\bigg|\bigg|_p^p\bigg|^{1/p}\\
    &\lesssim_{p,m}\E\bigg|\bigg|\sum_j\e_j(1+t_j^2\Delta)^{-m}f_j\bigg|\bigg|_p+\E\bigg|\bigg|\sum_j\e_jt_j^2\Delta(1+t_j^2\Delta)^{-m}f_j\bigg|\bigg|_p,
\end{align*}
where we again used the Kahane inequality. We can also see from the Kahane inequality that  $||\sqrt{I}||_p$ is equivalent to $\E||\sum_j\e_jt_j\nabla(1+t_j^2\Delta)^{-m}f_j||_p$. Since $\Delta$ has a bounded holomorphic functional calculus on $L^p(\M)$, it follows from \cite[Theorem 10.3.4]{ABS2} that $\{(1+t^2\Delta)^{-m},t>0\}$ and $\{t^2\Delta(1+t^2\Delta)^{-m},t>0\}$ are R-bounded on $L^p(\M)$. This, together with previous estimates, gives
\begin{equation*}
    \E\bigg|\bigg|\sum_j\e_jt_j\nabla(1+t_j^2\Delta)^{-m}f_j\bigg|\bigg|_p\lesssim\E\bigg|\bigg|\sum_j\e_jf_j\bigg|\bigg|_p
\end{equation*}
with a constant independent of $t_j$ and $f_j$. This shows that the set $\{t\nabla (1+t^2\Delta)^{-m}, t>0\}$ is R-bounded on $L^p(\M)$.
\end{proof}

\begin{remark}\label{RSF-SSF}
    The converse of the above theorem holds true by \cite[Theorem 4.1]{CO20}. Nevertheless, it is true for  $1<p< n^*$. This is because by \cite[Theorem 1.1]{BaSi} $S^{res,\nabla}_m$ is not bounded for $p\geq  n^*$ under the assumption that $2m< n^*$. Now, R-boundedness of the sets $\{\sqrt{t}\nabla (1+t\Delta)^{-m}:t>0\}$ and $\{\sqrt{t}\nabla e^{-t\Delta}:t>0\}$ is equivalent by \cite[Proposition 2.2]{CO20}, and Cometx-Ouhabaz have shown the equivalence between the R-boundedness of $\{\sqrt{t}\nabla e^{-t\Delta}:t>0\}$ with the boundedness of the square function $S^{exp,\nabla}$ in \cite[Theorem 3.1 and Theorem 4.1]{CO20}. On combining this together with Theorem~\ref{ResSF} we obtain equivalence of the boundedness of both $S^{exp,\nabla}$ and $S^{res,\nabla}_m$ on $L^p(\M)$ for $1<p<  n^*$ under the assumption that $2m+2\leq  n^*$.
\end{remark}
\subsection{Open problems}
Here we mention some open problems that one could think of solving on manifolds with ends or on a more general metric measure space.
\begin{enumerate}
    \item We have shown in this article the boundedness of the maximal operator $M^{res,\nabla}_m$ on $L^p(\M)$. The problem to check the boundedness of $M^{exp,\nabla}$ on $L^p(\M)$ is still open. The boundedness of these maximal operators is also open on other metric measure spaces.
    We believe that this problem is solvable if one can find Grigor'yan and Saloff-Coste type estimates \cite{GS1} on the gradient of the heat kernel, $\nabla e^{-t\Delta}$, on manifolds with ends. 
    %\item We believe that the problem in $(1)$ is solvable if one can find Grigor'yan and Saloff-Coste type estimates \cite{GS1} on the gradient of the heat kernel, $\nabla e^{-t\Delta}$, on manifolds with ends. 
    \item It is natural to expect that weak type $(1,1)$ can be included in the statement 
    of Proposition~\ref{prop4.2}, but we do not know how to verify it. 
\end{enumerate}
\bigskip

\noindent
{\bf Acknowledgements}

\medskip

\noindent Both authors were partly supported by the Australian Research Council (ARC) Discovery Grant  DP200101065.
The authors would like to thank the anonymous referee
for helpful remarks and corrections.

%\medskip

%\noindent
%{\bf Data availability} 
%Data sharing not applicable to this article as no datasets were generated or analysed in this study.

%\medskip

%\noindent
%{\bf Declarations}

%\medskip

%\noindent
%{\bf Conflict of interest} The authors declare that they have no conflict of interest.

\bibliographystyle{abbrv}
\bibliography{Bibliography}

\end{document}